\documentclass[11pt]{amsart}

\usepackage{amsfonts,amsthm,amsmath,enumerate,amssymb,latexsym,color,tcolorbox,tikz,mathrsfs,bm}
\usepackage[text={6in,9in},centering]{geometry}
\usepackage[backref=page,
            colorlinks,
            linkcolor=blue,
            anchorcolor=red,
            citecolor=red]{hyperref}

\graphicspath{{figures/}}

\usetikzlibrary{lindenmayersystems}

\theoremstyle{plain}
\newtheorem{theorem}{Theorem}[section]
\newtheorem{lemma}[theorem]{Lemma}
\newtheorem{corollary}[theorem]{Corollary}
\newtheorem{proposition}[theorem]{Proposition}

\newtheorem{question}[theorem]{Question}

\theoremstyle{definition}

\newtheorem{example}[theorem]{Example}

\theoremstyle{remark}
\newtheorem{remark}[theorem]{Remark}

\numberwithin{equation}{section}

\DeclareMathOperator{\dimh}{dim_H}
\DeclareMathOperator{\dist}{dist}

\DeclareMathOperator{\diam}{diam}
\DeclareMathOperator{\hdist}{d_H}
\DeclareMathOperator{\lip}{Lip}

\def\R{\mathbb R}
\def\Z{\mathbb Z}

\def\i{\bm i}
\def\j{\bm j}

\def\h{\mathcal H}

\title[From Lipschitz Embedding to Equivalence between Self-similar Sets]{From Lipschitz Embedding to Lipschitz Equivalence between Dust-like Self-similar Sets}
\author{Huo-Jun Ruan}
\address{School of Mathematical Sciences, Zhejiang University, Hangzhou 310058, China}
\email{ruanhj@zju.edu.cn}

\author{Jian-Ci Xiao$^{\ast}$}
\address{School of Mathematics, Nanjing University of Aeronautics and Astronautics, Nanjing 211106, China}
\email{jcxiao@nuaa.edu.cn}
\thanks{$^\ast$Corresponding author.}


\begin{document}

\begin{abstract}
    Let $K,F\subset\R^d$ be two dust-like self-similar sets sharing the same Hausdorff dimension. We consider when the mere existence of a Lipschitz embedding from $K$ to $F$ already implies their Lipschitz equivalence. Our main result is threefold: (1) if the Lipschitz image of $K$ intersects $F$ in a set of positive Hausdorff measure, then $K$ admits a Lipschitz surjection onto $F$; (2) if $F$ is in addition homogeneous, then the generating iterated function systems of $K, F$ should have algebraically dependent ratios and consequently, $K$ and $F$ are Lipschitz equivalent; (3) the Lipschitz equivalence can fail without the homogeneity assumption. This answers two questions in Balka and Keleti [Adv. Math. {\bf 446} (2024), 109669].
\end{abstract}
\maketitle

\section{Introduction}

Self-similar sets, which arise as attractors of self-similar iterated function systems (IFS), exhibit intricate geometric and topological properties that make them fundamental objects in the study of fractals. A key question in this context is to complete the Lipschitz classification of such sets, as Lipschitz equivalence preserves a number of crucial geometric and measure-theoretic properties, including fractal dimensions, regularity and metric scaling structures. Recall that $A,B\subset\R^d$ are said to be \emph{Lipschitz equivalent} if there exists a bijection $f: A \to B$ such that 
\[
    c^{-1}|x-y| \leq |f(x)-f(y)| \leq c|x-y|, \quad \forall x,y\in A
\]
for some constant $c>0$, that is, both of $f$ and its inverse $f^{-1}$ are Lipschitz maps.

Unlike topological equivalence, the Lipschitz equivalence requires a bounded geometric distortion, making the classification far more rigid and heavily reliant on the precise local patterns of self-similar sets. A natural starting point is to examine dust-like self-similar sets, where the absence of overlaps yields a clear geometric structure. Recall that a self-similar set $K$ generated by an iterated function system (IFS) $\Phi = \{\varphi_i\}_{i=1}^n$, where each $\varphi_i: \R^d \to \R^d$ is a contractive similitude, is called \emph{dust-like} when $n\geq 2$ and the images $\{\varphi_i(K)\}_{i=1}^n$ are pairwise disjoint. In this case, we also say that $\Phi$ satisfies the \emph{strong separation condition (SSC)}. For dust-like self-similar sets, it is proved that the Lipschitz equivalence reduces entirely to the algebraic compatibility of scaling ratios in their associated IFSs (see e.g.~\cite{RRX06}), thus simplifying many technical challenges. The situation becomes particularly tractable (see Lemma~\ref{lem:logdepimpliesequi}) when dealing with \emph{homogeneous} self-similar sets, where all similitudes in the associated IFS share a common contraction ratio. 

One of the foundational studies toward the algebraic direction was conducted by Falconer and Marsh~\cite{FM92}. They demonstrated that the Lipschitz equivalence between two dust-like self-similar sets requires specific multiplicative relationships among the corresponding scaling ratios, thereby revealing a deep connection to the algebraic properties of the IFSs. This line was further developed in~\cite{RRW12} where the first author, in collaboration with Rao and Wang, introduced a necessary condition for the Lipschitz equivalence called \emph{matchable condition}, which enables a complete  classification for two important cases: when the contraction vectors have full ranks, or when the IFSs consist of exactly two maps. Later, Rao and Zhang~\cite{RZ15} linked the classification with a higher-dimensional Frobenius problem, yielding criteria for equivalence when the scaling ratios are \emph{coplanar} in a certain sense. More recently, Xi and Xiong~\cite{XX21} gave a checkable characterization for systems with \emph{commensurable} ratios. For related results, see \cite{CP88, DS97, LuoLau13, Xi10} and the references therein.

Another direction investigates Lipschitz embeddings between dust-like self-similar sets and how such embeddability relates to the Lipschitz equivalence. Since Lipschitz maps cannot increase Hausdorff dimension, if $K,F$ are two dust-like self-similar sets with $\dimh K<\dimh F$ (where $\dimh$ denotes the Hausdorff dimension), then no Lipschitz maps can send $K$ onto $F$. On the other hand, if $\dimh K>\dimh F$, a Lipschitz embedding of $K$ onto $F$ always exists (see~\cite{BK24}). The critical case occurs when $\dimh K=\dimh F$. For this scenario, Deng, Wen, Xiong and Xi~\cite{DWXX11} proved that $K$ admits a bi-Lipschitz embedding into $F$ if and only if $K$ and $F$ are Lipschitz equivalent.

In 2024, Balka and Keleti~\cite{BK24} made significant progress. They discovered that in certain circumstances, the mere existence of a Lipschitz surjection from one dust-like self-similar set to another already implies the Lipschitz equivalence. Their investigation focused particularly on the case when $K$ (as the domain) is homogeneous. 

\begin{theorem}[\cite{BK24}]\label{thm:bk24}   
    Let $K,F\subset\R^d$ be two dust-like self-similar sets with the same Hausdorff dimension. If $K$ is homogeneous and there is a Lipschitz surjection from $K$ to $F$, then $K$ and $F$ are Lipschitz equivalent.
\end{theorem}
 
This result naturally raises the question of whether either the surjectivity condition or the homogeneity assumption can be weakened. More precisely, Balka and Keleti proposed the following two questions.

\begin{question}[{\cite[Section 8]{BK24}}]\label{que:1}
    Let $K,F\subset\R^d$ be two dust-like self-similar sets with the same Hausdorff dimension $s$. Denote by $\h^s$ the $s$-dimensional Hausdorff measure.
    \begin{enumerate}
        \item If $K$ is homogeneous and there is a Lipschitz map $f: K\to F$ with $\h^s(f(K))>0$, must $K$ and $F$ be Lipschitz equivalent?
        \item Under no homogeneity assumptions, if $K$ can be sent onto $F$ by a Lipschitz map, must $K$ and $F$ be Lipschitz equivalent? 
    \end{enumerate} 
\end{question}

The main aim of this paper is to address these questions. Our first result is as follows.

\begin{theorem}\label{thm:main1}
    Let $K,F\subset\R^d$ be two dust-like self-similar sets with the same Hausdorff dimension $s$. If there is a Lipschitz map $f: K\to\R^d$ with $\h^s(f(K)\cap F)>0$, then $K$ admits a Lipschitz surjection onto $F$.
\end{theorem}

The construction of such a surjection proceeds via three steps. We begin by determining a regular local region of $F$ where the map $f$ is not overly concentrated. Next, we apply a rescaling argument based on a standard density theorem in geometric measure theory. The desired surjection is then constructed by passing to an appropriate limit. Combining this result with Theorem~\ref{thm:bk24}, we immediately obtain an affirmative answer to Question~\ref{que:1}(1).

\begin{corollary}
    Let $K,F$ be as in Question~\ref{que:1}. If $K$ is homogeneous and there is a Lipschitz map $f: K\to F$ with $\h^s(f(K))>0$, then $K$ and $F$ are Lipschitz equivalent.
\end{corollary}

Regarding the second question, our findings present a dichotomy of results. On the one hand, we prove that the Lipschitz equivalence holds when $F$ (as the target space of the embedding) is homogeneous, establishing a converse to Balka and Keleti's original theorem. On the other hand, we demonstrate that the homogeneity assumption on $K$ or $F$ is indeed essential: without it, there exists an explicit example (see Example~\ref{exa:lastcountexa}) that provides a negative answer to Question~\ref{que:1}(2). We summarize these results as follows.

\begin{theorem}\label{thm:main2}
    Let $K,F\subset\R^d$ be two dust-like self-similar sets with the same Hausdorff dimension $s$.
    \begin{enumerate}
        \item If $F$ is homogeneous and there is a Lipschitz map $f: K\to F$ with $\h^s(f(K))>0$, then $K$ and $F$ are Lipschitz equivalent.
        \item Assume that $K$ can be sent onto $F$ by a Lipschitz map. Without the homogeneity assumption, $K$ and $F$ need not to be Lipschitz equivalent.
    \end{enumerate}
\end{theorem}

Here is an outline of the key steps of the proof. For part (1), it suffices to build the algebraic dependence between the corresponding contraction ratios (see Lemma~\ref{lem:logdepimpliesequi}). Note that Theorem~\ref{thm:main1} guarantees a Lipschitz surjection from $K$ to $F$. From this, we construct an alternative map which is almost injective and enjoys a crucial measure-preserving property. It is worth pointing out that the measure-preserving property is enough to derive the algebraic dependence when $K$ is homogeneous, thus providing an alternative way to recover Theorem~\ref{thm:bk24}. When $F$ is homogeneous instead, the argument admits a certain ``reversal'' (in some sense) once we prove that the image of every cell of $K$ under the alternative map contains a cell of $F$ of comparable size. For part (2), we rely on a powerful mass decomposition lemma in~\cite{XX21}. Using these, we explicitly construct a Lipschitz surjection between two dust-like self-similar sets that fail to be Lipschitz equivalent, thereby completing the proof.

Our results are also related to a notable conjecture by Feng, Huang and Rao~\cite{FHR14}, which states that the existence of an affine embedding from $K$ into $F$ may enforce the algebraic dependence between their associated ratios even without assuming that $\dimh K=\dimh F$. Significant progress has been made toward this conjecture, see~\cite{AHW24,EKM10,FX18}. In particular, a resolution of the special case when $\dimh K=\dimh F$ was recently obtained as a corollary in an ongoing work~\cite{RX25} of the second author joint with Rao. The proofs of Theorems~\ref{thm:bk24}, \ref{thm:main1} and~\ref{thm:main2}(1) reveal that if in addition that either $K$ or $F$ is homogeneous, the algebraic dependence follows already from the existence of a nontrivial Lipschitz embedding---a weaker condition than the affinity requirement. However, as shown in~\cite{DWXX11}, such dependence is not guaranteed without the dimensional restriction.

The paper is organized as follows. Section 2 introduces some necessary preliminaries. Section 3 establishes our notational conventions and presents the proof of Theorem~\ref{thm:main1}. Building on these foundations, Sections 4 and 5 develop the construction of a well-behaved Lipschitz surjection and present the proof of Theorem~\ref{thm:main2}(1). Finally, Section 6 provides an explicit example that demonstrates the second part of Theorem~\ref{thm:main2}.

\paragraph{{\bf Notation}} We write $X\lesssim Y$ when $X\leq CY$ for some constant $C>0$ and write $X\approx Y$ when $X\lesssim Y$ and $Y\lesssim X$ hold simultaneously. For $A,B\subset\R^d$, we denote by $\diam(A)$ the diameter of $A$ and write $\dist(A,B) := \inf\{|a-b|: a\in A, b\in B\}$.

\section{Preliminaries}

We begin with a classical Lipschitz extension theorem of Kirszbraun.

\begin{lemma}[\cite{Kir34}]\label{lem:kirszbraun}
    Let $E\subset \R^n$ and let $f: E\to \R^m$ be any Lipschitz map. Then $f$ can be extended to a Lipschitz map $g:\R^n\to\R^m$ with $\lip(g)=\lip(f)$, where $\lip(f)$ denotes the Lipschitz constant of $f$, that is, 
    \[
        \lip(f) = \sup\Big\{ \frac{|f(x)-f(y)|}{|x-y|}: x\neq y \in E \Big\}.
    \]
\end{lemma}

In what follows, we will fix two self-similar sets $K,F\subset\R^d$ with the same Hausdorff dimension $s$, generating by self-similar IFSs $\Phi=\{\varphi_i\}_{i\in I}$ and $\Psi=\{\psi_j\}_{j\in J}$ satisfying the SSC, respectively. Without loss of generality, we always assume that both of $K,F$ contain the origin and have diameter $1$. 

One of the principal advantages of self-similarity is that it allows us, via an appropriate limiting process, to extract universal geometric or analytic features that persist across scales. A fundamental step in our subsequent arguments is as follows. 

\begin{lemma}\label{lem:takinglimits}
    Let $P\geq 1$, $\gamma,L>0$ and let $[u,v]\subset(0,\infty)$ be an interval. Set $\mathcal{K}=\mathcal{K}(P,\gamma,u,v)$ to be the collection of all sets of the form $\bigcup_{t=1}^p (c_{t}O_{t}K+a_{t})$, where $1\leq p\leq P$, $c_{t}\in[u,v]$, $O_{t}$ is a $d\times d$ orthogonal matrix and $a_{t}\in\R^d$, such that 
    \[ 
        \min_{t\neq t'}\dist(c_{t}O_{t}K+a_{t}, c_{t'}O_{t'}K+a_{t'})\geq \gamma. 
    \]
    For any sequence $\{E_n\}_{n=1}^\infty\subset\mathcal{K}$ and Lipschitz maps $f_n: E_n \to F$ with $\sup_n\lip(f_n)\leq L$, there exists $E\in\mathcal{K}$ and a Lipschitz map $g: E \to F$ with $\lip(g)\leq L$ such that 
    \[
        \h^s(E) = \lim_{j\to\infty} \h^s(E_{n_j}) \quad\text{and}\quad \h^s(g(E)) \geq \limsup_{j\to\infty} \h^s(f_{n_j}(E_{n_j}))
    \]
    for some subsequence $\{n_j\}$.
\end{lemma}
\begin{proof}
    Since $E_n\in\mathcal{K}$, we can write it as $E_n = \bigcup_{t=1}^{p_n} (c_{n,t}O_{n,t}K+a_{n,t})$. Passing to a subsequence if necessary, we may assume that $p_n\equiv p$ for some $1\leq p\leq P$. Put $\widetilde{\gamma}:=v+L^{-1}+\gamma$ and $b_t:=(3t\widetilde{\gamma},0,\ldots,0)\in\R^d$ for $1\leq t\leq p$.  For $n\geq 1$, it is easy to see that when $t\neq t'$,
    \begin{align}
        \dist(c_{n,t}O_{n,t}K \, & +b_t, c_{n,t'}O_{n,t'}K+b_{t'}) \notag \\
        &\geq |b_t-b_{t'}|-\diam(c_{n,t}O_{n,t}K+b_t)-\diam(c_{n,t'}O_{n,t'}K+b_{t'}) \notag \\
        &\geq 3\widetilde{\gamma}-v-v \notag \\
        &> \widetilde{\gamma} > \max\{L^{-1},\gamma\}. \label{eq:originaldiscond}
    \end{align}
    Shifting $f$ slightly on each piece, we obtain a new sequence of maps given by
    \begin{align*}
        g_n : \bigcup_{t=1}^p (c_{n,t}O_{n,t}K+b_t) &\to F \\
        c_{n,t}O_{n,t}x+b_t &\mapsto f_n(c_{n,t}O_{n,t}x+a_{n,t}),
    \end{align*}
    which is well defined because the domain is a disjoint union (recall~\eqref{eq:originaldiscond}). Clearly,
    \begin{equation}\label{eq:gnequalsfnkn}
        g_n\Big( \bigcup_{t=1}^p (c_{n,t}O_{n,t}K+b_t) \Big) = f_n\Big( \bigcup_{t=1}^p (c_{n,t}O_{n,t}K+a_{n,t}) \Big) = f_n(E_n).
    \end{equation}

    Fix any $n\geq 1$. Let $x,y\in\bigcup_{t=1}^p (c_{n,t}O_{n,t}K+b_t)$, say $x=c_{n,t}O_{n,t}\alpha+b_t$ and $y=c_{n,t'}O_{n,t'}\beta+b_{t'}$. If $t=t'$ then
    \begin{align*}
        |g_n(x)-g_n(y)| &= |f_n(c_{n,t}O_{n,t}\alpha+a_{n,t}) - f_n(c_{n,t}O_{n,t}\beta+a_{n,t})| \\
        &\leq \lip(f_n)\cdot c_{n,t}|\alpha-\beta| \\
        &\leq L|x-y|.
    \end{align*} 
    If $t\neq t'$ then  
    \begin{align*}
        |g_n(x)-g_n(y)| &\leq 1 && \text{(since $\diam(F)=1$)} \\
        &\leq L\dist(c_{n,t}O_{n,t}K+b_t, c_{n,t'}O_{n,t'}K+b_{t'}) &&\text{(by~\eqref{eq:originaldiscond})} \\
        &\leq L|x-y|.
    \end{align*}
    Thus $g_n$ is Lipschitz as well with $\lip(g_n)\leq L$.

    Passing to a subsequence and rearranging if necessary, we may assume that for every $1\leq t\leq p$, both of the sequences $\{c_{n,t}\}_n, \{O_{n,t}\}_n$ are convergent, say $c_{n,t}\to c_t\in [u,v]$ and $O_{n,t}\to O_t$. Writing $E:=\bigcup_{t=1}^p (c_tO_tK+b_t)$, we have by~\eqref{eq:originaldiscond} and the convergence that $E\in\mathcal{K}$. Also, 
    \[
        \h^s(E) = \sum_{t=1}^p c_t^s\h^s(K) = \lim_{n\to\infty} \sum_{t=1}^p c_{n,t}^s\h^s(K) = \lim_{n\to\infty} \h^s(E_n).
    \]

    To find a limit of $g_n$, we first pick a large $R>0$ so that $\bigcup_{t=1}^p (c_{n,t}O_{n,t}K+b_t)\subset B(0,R)$ for all $n\geq 1$. By Lemma~\ref{lem:kirszbraun}, each $g_n$ extends to a Lipschitz map $\widetilde{g}_n: B(0,R) \to \R^d$ with $\lip(\widetilde{g}_n)=\lip(g_n)\leq L$. In particular, $\{\widetilde{g}_n\}$ is equicontinuous. Furthermore, since $\widetilde{g}_n(B(0,R))\cap F\neq\varnothing$ and
    \[
        \diam(\widetilde{g}_n(B(0,R))) \leq \lip(\widetilde{g}_n)\cdot\diam(B(0,R)) \leq 2LR,
    \]
    $\widetilde{g}_n(B(0,R))$ is contained in the $2LR$-neighborhood of $F$ for all $n$, that is, $\{\widetilde{g}_n\}$ is pointwise bounded. Then by Arzel\'a-Ascoli theorem (e.g., see~\cite[Corollary 1.8.25]{Tao10}), there exists a subsequence $\{\widetilde{g}_{n_j}\}_{j=1}^\infty$ which converges uniformly to a limit map, say $g$. Clearly, $g$ is also Lipschitz with $\lip(g)\leq L$.

    Finally, since $\widetilde{g}_n$ extends $g_n$,
    \begin{equation}\label{eq:tildegextendsg}
        \widetilde{g}_n\Big( \bigcup_{t=1}^p (c_{n,t}O_{n,t}K+b_t) \Big) = g_n\Big( \bigcup_{t=1}^p (c_{n,t}O_{n,t}K+b_t) \Big) \subset F, \quad \forall n\geq 1.
    \end{equation}
    Then we have by the compactness of $F$ and all the convergences that 
    \[
        g(E) = g\Big( \bigcup_{t=1}^p (c_{t}O_{t}K+b_t) \Big) \subset F
    \]
    and for every $\varepsilon>0$, 
    \[
        \widetilde{g}_{n_j}\Big( \bigcup_{t=1}^p (c_{n_j,t}O_{n_j,t}K+b_t) \Big) \subset \mathcal{N}_{\varepsilon}(g(E))\cap F \quad\text{ for all large $j$},
    \]
    where $\mathcal{N}_\varepsilon(\cdot)$ denotes the open $\varepsilon$-neighborhood. 
    Note that the restriction of $\h^s$ on $F$ is a Radon measure and, since $g(E)\subset F$,
    \[
        \h^s(g(E)) =\h^s(g(E)\cap F) = \lim_{\varepsilon\to 0} \h^s(\mathcal{N}_\varepsilon(g(E))\cap F).
    \]
    Therefore, 
    \begin{align*}
        \h^s(g(E)) &\geq \limsup_{j\to\infty} \h^s\Big( \widetilde{g}_{n_j}\Big( \bigcup_{t=1}^p (c_{{n_j},t}O_{{n_j},t}K+b_t) \Big) \Big) \\
        &= \limsup_{j\to\infty} \h^s\Big( g_{n_j}\Big( \bigcup_{t=1}^p (c_{{n_j},t}O_{{n_j},t}K+b_t) \Big) \Big) \\
        &= \limsup_{j\to\infty} \h^s(f_{n_j}(E_{n_j})).  && \text{(by~\eqref{eq:gnequalsfnkn})}.
    \end{align*}
    This completes the proof.
\end{proof}

\begin{remark}\label{rem:eventuallydense}
    Under the assumptions of the above lemma, if $f_n(E_n)$ is eventually dense in $F$, i.e., $f_n(E_n)$ converges to $F$ under the Hausdorff distance $\hdist$, then $g(E)=F$. This is  because $f_{n_j}(E_{n_j})=\widetilde{g}_{n_j}\big(\bigcup_{t=1}^p (c_{n_j,t}O_{n_j,t}K+b_t)\big)\to g(E)$  under $\hdist$ (recall~\eqref{eq:gnequalsfnkn}, ~\eqref{eq:tildegextendsg} and all the previous convergences).  
\end{remark}

Write $r_i$ (resp. $\lambda_j$) to be the contraction ratio of $\varphi_i$, $i\in I$ (resp. $\psi_j$, $j\in J$). When either $\Phi$ or $\Psi$ is homogeneous, the algebraic dependence of the corresponding ratios is in fact sufficient to establish the Lipschitz equivalence between $K$ and $F$. While this fact is implicitly contained in the work of Balka and Keleti~\cite{BK24}, we state it explicitly as a separate lemma for the reader's convenience.

\begin{lemma}\label{lem:logdepimpliesequi}
    Assume that $r_i\equiv r$ for some $0<r<1$. If $\frac{\log\lambda_j}{\log r}\in\mathbb{Q}$ for all $j\in J$, then $K, F$ are Lipschitz equivalent.
\end{lemma}
\begin{proof}
    Write $\lambda_j=r^{\alpha_j}$, where $\alpha_j\in\mathbb{Q}$ for $j\in J$. The SSC implies that $\#I\cdot r^s=1$ and $\sum_{j\in J}\lambda_j^s=1$. If $\#I=m^k$ for some integers $m,k>1$, then $K$ is Lipschitz equivalent to any self-similar set $E=\bigcup_{t=1}^m g_t(E)$ with the SSC such that each $g_t$ has similarity ratio $r^{1/k}$ (e.g., see~\cite{RRW12}). So we may assume that $k=1$ and $m$ is not of the form $p^\ell$ where $p,\ell$ are integers strictly larger than $1$. Then $r^{s}=m^{-1}$ and thus 
    \[
        1 = \sum_{j\in J} \lambda_j^s = \sum_{j\in J} r^{\alpha_js} = \sum_{j\in J} m^{-\alpha_j}.
    \]
    It follows from~\cite[Lemma 7.1]{BK24} that $\alpha_j$ are all integers. Then by~\cite[Theorem 7.7]{BK24} or~\cite[Theorem 1.5]{RRW12}, $K$ and $F$ are Lipschitz equivalent.
\end{proof}

\section{Proof of Theorems~\ref{thm:main1}}

In this section, we present the proof of Theorem~\ref{thm:main1}. Regarding $\Phi=\{\varphi_i\}_{i\in I}$ and $\Psi=\{\psi_j\}_{j\in J}$, we adopt the following customary notations in the rest of this paper.

Let $I^*:=\bigcup_{n=1}^\infty I^n$ be the set of all finite words over the alphabet $I$ with the convention that $I^0=\{\vartheta\}$, where $\vartheta$ denotes the empty word. For $n\geq 0$ and $\i\in I^n$, we call $n$ the \emph{length} of $\i$ and denote it by $|\i|$. When $n\geq 1$, we write for $\i=i_1\cdots i_n\in I^n$ that 
\[
    \varphi_{\i}:=\varphi_{i_1}\circ\cdots\circ\varphi_{i_n},\quad K_{\i}:=\varphi_{\i}(K),\quad r_{\i}:=r_{i_1}\cdots r_{i_n}, \quad \i^{-}:=i_1\cdots i_{n-1},
\]
with the convention that $\i^{-}=\vartheta$ if $n=1$; when $n=0$, we write $\varphi_{\vartheta}$ to be the identity map and let $K_\vartheta:=K$ and $r_{\vartheta}:=1$. For any $\i,\i'\in I^*$, $\i\wedge \i'$ denotes the longest common prefix of $\i$ and $\i'$. Analogously, for the alphabet $J$, we write $J^*:=\bigcup_{n=1}^\infty J^n$, $J^0:=\{\vartheta\}$ and adopt the notation $\psi_{\j}$, $F_{\j}$, $\lambda_{\j}$, $\j^{-}$ and $\j\wedge\j'$ in the same manner as above.

Furthermore, let $\underline{r}:=\min_{i\in I} r_i$, $\underline{\lambda}:=\min_{j\in J} \lambda_j$, $\Delta_K:=\min_{i\neq i'\in I}\dist(K_i,K_{i'})$ and $\Delta_F:=\min_{j\neq j'\in J}\dist(F_j, F_{j'})$. The SSC ensures that $\Delta_K$ and $\Delta_F$ are both positive.

Our first observation is a reduction of the problem to its embedding case.

\begin{lemma}\label{lem:lipschitzembedding}
    For any Lipschitz map $f: K\to\R^d$ with $f(K)\cap F\neq\varnothing$, there is a Lipschitz map $g:K\to F$ with $g(K)=f(K)\cap F$.
\end{lemma}
\begin{proof}
	We begin by constructing a map $g:K\to F$ with $g(K)=f(K)\cap F$. 
    For any $\i\in I^*$ such that $f(K_{\i})\cap F=\varnothing$ but $f(K_{\i^-})\cap F\neq\varnothing$, choose an arbitrary point $y_{\i}$ in $f(K_{\i^-})\cap F$ and define $g(K_{\i})=\{y_{\i}\}$. So for those $x\in K$ with $f(x)\notin F$, $g(x)$ is defined as $y_\omega$, where $\omega\in I^*$ is the unique word such that $x\in K_{\omega}$, $f(K_{\omega})\cap F=\varnothing$ and $f(K_{\omega^-})\cap F\neq\varnothing$. Note that the existence of such an $\omega$ is guaranteed by the closedness of $F$. For $x\in K$ with $f(x)\in F$, we simply set $g(x)=f(x)$. In this way, the map $g: K\to F$ is well defined and it is clear that $g(K)=f(K)\cap F$. 
    
    We claim that $\diam(g(K_{\i}))\leq \diam(f(K_{\i}))$ for all $\i\in I^*$. Indeed, if $f(K_{\i})\cap F=\varnothing$, then $g(K_{\i})$ is a singleton and the inequality holds trivially. If $f(K_{\i})\cap F\neq\varnothing$, then by the definition of $g$,
    \[
        g(K_{\i}) \subset \bigcup_{i\in I: f(K_{\i i})\cap F\neq\varnothing} f(K_{\i i}) \subset f(K_{\i}),
    \]
    which also implies the desired inequality.
    
    It remains to prove that $g$ is Lipschitz. For any distinct $x,x'\in K$, let $\omega\in I^*\cup I^0$ be the longest word such that $x,x'\in K_{\omega}$. So 
    \[
        |x-x'| \geq \min_{i\neq i'} \dist(K_{\omega i}, K_{\omega i'}) = r_{\omega}\Delta_K = \Delta_K\diam(K_{\omega}).
    \] 
    Then by the previous claim,
    \begin{align*}
        |g(x)-g(x')| \leq \diam(g(K_{\omega})) &\leq \diam(f(K_{\omega})) \\
        &\leq \lip(f)\cdot\diam(K_{\omega}) \leq \lip(f)\cdot\Delta_K^{-1}|x-x'|.
    \end{align*}
    So $g$ is Lipschitz.
\end{proof}

The proof of Theorem~\ref{thm:main1} then reduces to establishing the next result.

\begin{theorem}\label{thm:main1reduction}
    If there is a Lipschitz map $f: K\to F$ with $\h^s(f(K))>0$, then $K$ admits a Lipschitz surjection onto $F$.
\end{theorem}

In the rest of this section, let us fix such a Lipschitz embedding $f:K\to F$. For notational simplicity, write $L:=\lip(f)$.

\subsection{Localization for rescaling}

Fix any small $0<\delta<1$ and set 
\begin{equation}\label{eq:originalinjn}
    \left\{\begin{array}{l} I_n := \{\i\in I^*: r_{\i} \leq \delta^n< r_{\i^-}\}, \\ J_n:= \{\j\in J^*: \lambda_{\j} \leq \delta_n < \lambda_{\j^-}\}, \end{array}\right.
\end{equation}
where $\delta_n:=\frac{L}{\Delta_F}\cdot\delta^n$ is a constant multiple of $\delta^n$. The next lemma records several basic facts that will be used frequently later. Recall that $\underline{r}:=\min_{i\in I}r_i$ and $\underline{\lambda}:=\min_{j\in J}\lambda_j$.

\begin{lemma}\label{lem:factsinq2}
    Let $n\geq 1$.
    \begin{enumerate}
        \item We have $\underline{r}\delta^n<r_{\i}\leq \delta^n$ for $\i\in I_n$ and $\underline{\lambda}\delta_n < \lambda_{\j} \leq \delta_n$ for $\j\in J_n$. In particular, $\#I_n\approx \#J_n\approx \delta^{-ns}$, where the implicit constants are independent of $n$.
        \item For every distinct $\i,\i'\in I_n$, $\dist(K_{\i}, K_{\i'})>\Delta_K\delta^n$; for every distinct $\j,\j'\in J_n$, $\dist(F_{\j}, F_{\j'})>L\delta^n$.
        \item For every $\i\in I_n$, there is a unique word $\j\in J_n$ such that $f(K_{\i})\subset F_{\j}$. 
    \end{enumerate}
\end{lemma}
\begin{proof}
    The first statement follows directly from~\eqref{eq:originalinjn} and the fact that $\sum_{\i\in I_n}r_{\i}^s=\sum_{\j\in J_n}\lambda_{\j}^s=1$.
    For (2), note that 
    \[
        \dist(K_{\i}, K_{\i'}) \geq r_{\i\wedge \i'}\Delta_K \geq r_{\i^-}\Delta_K > \Delta_K\delta^n;
    \]
    similarly,
    \[
        \dist(F_{\j}, F_{\j'}) \geq  \lambda_{\j^-}\Delta_F> \delta_n\Delta_F  = L\delta^n.
    \]
    For (3), just note that when $\i\in I_n$,
    \begin{align*}
        \diam(f(K_{\i})) \leq L\diam(K_{\i}) = Lr_{\i} \leq L\delta^n < \min_{\j\neq\j'\in J_n}\dist(F_{\j}, F_{\j'}),
    \end{align*}
    where the last inequality comes from (2). 
\end{proof}

\begin{lemma}\label{lem:densitythm}
    For $n\geq 1$ and $y\in F$, let $\j_n(y)$ be the unique word in $J_{n}$ such that $y\in F_{\j_n(y)}$. Then 
    \[
        \lim_{n\to\infty} \frac{\h^s(f(K)\cap F_{\j_n(y)})}{\h^s(F_{\j_n(y)})} = 1, \quad \h^s\text{-a.e. } y\in f(K).
    \]
\end{lemma}
\begin{proof}
    Since the restriction of $\h^s$ on $F$ is a Radon measure, by a standard density theorem (see~\cite[Corollary 2.14]{Mat85}), 
    \begin{equation}\label{eq:densitythm-1}
        \lim_{\rho\to 0} \frac{\h^s(f(K)\cap B(y,\rho))}{\h^s(F\cap B(y,\rho))} =\lim_{\rho\to 0} \frac{\h^s(F\cap f(K)\cap B(y,\rho))}{\h^s(F\cap B(y,\rho))} = 1
    \end{equation}
    for $\h^s$-a.e. $y\in f(K)$, where the first equality follows from $f(K)\subset F$. 
    Fix any $y$ satisfying \eqref{eq:densitythm-1}. For any $0<\varepsilon<1$, we have for all small $\rho$ that
    \begin{equation}\label{eq:density1}
        \h^s(f(K)\cap B(y,\rho)) \geq (1-\varepsilon)\h^s(F\cap B(y,\rho)).
    \end{equation}

    Pick the smallest $n$ so that $\delta_n\leq \rho$. Since $\j_n(y)\in J_{n}$, $\diam(F_{\j_n(y)})=\lambda_{\j_n(y)}\leq\delta_n$ and thus $F_{\j_n(y)}\subset B(y,\rho)$. On the other hand, by Lemma~\ref{lem:factsinq2}(1) and the minimality of $n$, we have $\rho/C_1\leq \delta_n\leq C_1\rho$ for some constant $C_1>1$. 
    By Lemma~\ref{lem:factsinq2}(2), one can find a constant $C_2>0$ such that for all $n\geq 1$ and $\j\in J_n$, there exists $x_{\j}\in F_{\j}$ such that $\{B(x_{\j},C_2 \rho)\}_{\j\in J_n}$ are mutually disjoint. Thus
    \begin{align*}
    	\#\big\{\j\in J_n: F_{\j}\cap B(y,\rho)\not=\emptyset\big\} & \leq 
    	\#\big\{\j\in J_n: B(x_{\j},C_2\rho) \subset B\big(y,(1+C_1+C_2)\rho\big)\big\} \\
    	&\leq \Big(\frac{1+C_1+C_2}{C_2}\Big)^d.
    \end{align*}
    As a result, there is a constant $C>0$ (independent of $y, n, \varepsilon$) such that
    \begin{equation}\label{eq:massofballandfjn}
        \h^s(F\cap B(y,\rho)) \leq \sum_{\j\in J_{n}: F_{\j}\cap B(y,\rho)\neq\varnothing} \h^s(F_{\j}) \leq C\h^s(F_{\j_n(y)}).
    \end{equation}
    Recalling that $f(K)\subset F$ and $F_{\j_n(y)}\subset B(y,\rho)\cap F$,
    we have
    \begin{align*}
        \h^s(f(K)\cap B(y,\rho)) &\leq \h^s(F\cap B(y,\rho))-\h^s((F\cap B(y,\rho))\setminus f(K)) \\
        &\leq \h^s(F\cap B(y,\rho)) - \h^s(F_{\j_n(y)}\setminus f(K)) \\
        &= \h^s(F\cap B(y,\rho)) - \frac{\h^s(F_{\j_n(y)}\setminus f(K))}{\h^s(F_{\j_n(y)})}\cdot \h^s(F_{\j_n(y)}) \\
        &\leq \Big( 1-C^{-1}\cdot\frac{\h^s(F_{\j_n(y)}\setminus f(K))}{\h^s(F_{\j_n(y)})} \Big)\h^s(F\cap B(y,\rho)),
    \end{align*}
    where the last inequality follows from~\eqref{eq:massofballandfjn}. By~\eqref{eq:density1}, we immediately obtain 
    \[
        \frac{\h^s(F_{\j_n(y)}\setminus f(K))}{\h^s(F_{\j_n(y)})}\leq C\varepsilon.
    \]
    Since $\varepsilon$ is arbitrary, the proof is finished.
\end{proof}

To obtain an effective rescaling, we also need to determine a local region within $F$ where the map $f$ is not overly concentrated.

\begin{proposition}\label{prop:countingprop}
    There is an integer $Q\geq 1$ satisfying the following property. For any $0<\varepsilon<1$, we can find some $n$ and $\j\in J_n$ such that $\frac{\h^s(f(K) \cap F_{\j})}{\h^s(F_{\j})}>1-\varepsilon$ and 
    \[
        \#\{\i\in I_n: f(K_{\i})\subset F_{\j}\} = \#\{\i\in I_n: f(K_{\i})\cap F_{\j}\neq\varnothing\} \leq Q.
    \]
\end{proposition}
\begin{proof}
    The first equality is immediately from Lemma~\ref{lem:factsinq2}(3) as the two sets are identical. It remains to verify the cardinality.   
    Let $0<\varepsilon<1$. By Lemma~\ref{lem:densitythm}, there exists $E\subset f(K)$ with $\h^s(E)=\h^s(f(K))$ such that for all $y\in E$, 
    \[
        \h^s(f(K)\cap F_{\j_n(y)}) > (1-\varepsilon)\h^s(F_{\j_n(y)})
    \]
    when $n$ is large enough. This naturally gives us a Vitali covering of $E$ consisting of specific cells of $F$. By a version of Vitali covering theorem (see~\cite[Theorem 1.10]{Fal86}), we can find $\{y_t\}_{t=1}^\infty\subset E$ and $\{n_t\}_{t=1}^\infty$ such that $\h^s(E\setminus\bigcup_{t=1}^\infty F_{\j_{n_t}(y_t)})=0$, where $\{F_{\j_{n_t(y_t)}}\}_t$ is a disjoint family and 
    \begin{equation}\label{eq:counting1}
        \h^s(f(K)\cap F_{\j_{n_t}(y_t)}) > (1-\varepsilon)\h^s(F_{\j_{n_t}(y_t)}), \quad \forall t\geq 1.
    \end{equation}
    For convenience, below we abbreviate $\j_{n_t}:=\j_{n_t}(y_t)$.

    Pick any $p\geq 1$ such that $\h^s( \bigcup_{t=1}^p F_{\j_{n_t}} ) \geq \h^s(E)/2=\h^s(f(K))/2$. So
    \begin{equation}\label{eq:choiceofp}
        \sum_{t=1}^p\delta_{n_t}^s \geq  \h^s(F)^{-1} \sum_{t=1}^p \lambda_{n_t}^s\h^s(F) =  \h^s(F)^{-1} \sum_{t=1}^p \h^s(F_{\j_{n_t}}) \geq \frac{\h^s(f(K))}{2\h^s(F)}.
    \end{equation}
    Pick an integer $P\geq 2\max\{n_1,\ldots,n_p\}$ and a constant $C>\underline{\lambda}^{-2s}$. Set 
    \begin{equation}\label{eq:defofmt}
        \mathcal{M}_t := \Big\{ \eta\in J_{P}: F_\eta\subset F_{\j_{n_t}}, \frac{\h^s(f(K)\cap F_\eta)}{\h^s(F_\eta)}>1-C\varepsilon \Big\}, \quad 1\leq t\leq p
    \end{equation}
    and $\mathcal{M}:=\bigcup_{t=1}^p \mathcal{M}_t$. Since $F_{\j_{n_1}},\ldots,F_{\j_{n_p}}$ are disjoint, $\mathcal{M}_1,\ldots,\mathcal{M}_p$ are disjoint. We claim that $\#\mathcal{M}\gtrsim \delta_P^{-s}$.

    To this end, fix any $1\leq t\leq p$. From $\j_{n_t}\in J_{n_t}$ and \eqref{eq:counting1}, 
    \begin{equation}\label{eq:countingprop-new1}
	  \varepsilon\delta_{n_t}^s\h^s(F) \geq \varepsilon\lambda_{\j_{n_t}}^s\h^s(F)= \varepsilon\h^s(F_{\j_{n_t}})> \h^s(F_{\j_{n_t}}\setminus f(K)).
    \end{equation}
    On the other hand, by~\eqref{eq:defofmt}, for any $\eta\in J_P\setminus \mathcal{M}_t$ with $F_\eta\subset F_{\j_{n_t}}$, we have 
    \[
       \h^s(F_\eta\setminus f(K))\geq C\varepsilon \h^s(F_\eta)\geq C\varepsilon \cdot (\underline{\lambda}\delta_P)^s\h^s(F).
    \]
    So
    \begin{align*}
      \h^s(F_{\j_{n_t}}\setminus f(K))&\geq \sum_{\eta\in J_{P}\setminus\mathcal{M}_t: F_\eta\subset F_{\j_{n_t}}} \h^s(F_\eta\setminus f(K)) \\
      	&\geq \#\{\eta\in J_{P}\setminus\mathcal{M}_t: F_\eta\subset F_{\j_{n_t}}\}\cdot C\varepsilon \cdot (\underline{\lambda}\delta_P)^s\h^s(F).
    \end{align*}
    Combining this with \eqref{eq:countingprop-new1}, we see that
    \[
        \#\{\eta\in J_{P}\setminus\mathcal{M}_t: F_\eta\subset F_{\j_{n_t}}\} \leq C^{-1}\underline{\lambda}^{-s}\delta_{n_t}^s\delta_P^{-s}
    \]
    and hence
    \begin{align}
        \#\mathcal{M}_t &= \#\{\eta\in J_{P}: F_\eta\subset F_{\j_{n_t}}\} - \#\{\eta\in J_{P}\setminus\mathcal{M}_t: F_\eta\subset F_{\j_{n_t}}\} \notag \\
        &\geq \frac{\h^s(F_{\j_{n_t}})}{\max_{\eta\in J_{P}}\h^s(F_{\eta})} - C^{-1}\underline{\lambda}^{-s}\delta_{n_t}^s\delta_P^{-s} \notag  \\ 
        &\geq \frac{(\underline{\lambda}\delta_{n_t})^s}{\delta_P^s} - C^{-1}\underline{\lambda}^{-s}\delta_{n_t}^s\delta_P^{-s} \notag \\
        &= (\underline{\lambda}^s-C^{-1}\underline{\lambda}^{-s})\delta_{n_t}^s\delta_P^{-s} =: C_1\delta_{n_t}^s\delta_P^{-s}, \label{eq:counting2}
    \end{align}
    where $C_1>0$ since $C>\underline{\lambda}^{-2s}$. Letting $C_2:=\frac{C_1\h^s(f(K))}{2\h^s(F)}$, we know from the disjointness of $\{\mathcal{M}_t\}_{t=1}^p$, the estimates~\eqref{eq:counting2} and~\eqref{eq:choiceofp} that
    \begin{align*}
        \#\mathcal{M} = \sum_{t=1}^p \#\mathcal{M}_t \geq C_1\delta_{P}^{-s}\sum_{t=1}^p\delta_{n_t}^s \geq C_2\delta_P^{-s},
    \end{align*}
    as claimed.

    Lemma~\ref{lem:factsinq2}(3) also implies that $\{\{\i\in I_P: f(K_{\i})\cap F_{\eta}\neq\varnothing\}\}_{\eta\in\mathcal{M}}$ is a disjoint family. By the pigeonhole principle, there must exist a word $\eta_0\in\mathcal{M}$ such that 
    \[
        \#\{\i\in I_P: f(K_{\i})\cap F_{\eta_0}\neq\varnothing\} \leq \frac{\#I_P}{\#\mathcal{M}}.
    \]
    By Lemma~\ref{lem:factsinq2}(1), $\#I_P\approx\delta^{-Ps}$. Hence
    \[
        \#\{\i\in I_P: f(K_{\i})\cap F_{\eta_0}\neq\varnothing\} \lesssim \frac{\delta^{-Ps}}{C_2\delta_P^{-s}} = \frac{\delta^{-Ps}}{C_2}\cdot \Big( \frac{L\delta^P}{\Delta_F} \Big)^s = \frac{L^s}{C_2\Delta_F^s},
    \]
    Combining this with \eqref{eq:defofmt}, we prove the proposition.
\end{proof}

\subsection{Proof of Theorem~\ref{thm:main1reduction}}

Now let us prove Theorem~\ref{thm:main1reduction}. Applying Proposition~\ref{prop:countingprop} to $\varepsilon=n^{-1}$, we find for all $n\geq 1$ an integer $k_n\geq 1$ and a word $\j_n\in J_{k_n}$ such that 
\begin{equation}\label{eq:thesequence}
    \h^s(F_{\j_n}\cap f(K)) > (1-n^{-1})\h^s(F_{\j_n}),
\end{equation}
whereas 
\begin{equation*}
    P_n:= \#\{\i\in I_{k_n}: f(K_{\i})\cap F_{\j_n}\neq\varnothing\} = \#\{\i\in I_{k_n}: f(K_{\i})\subset F_{\j_n}\}
\end{equation*}
are uniformly bounded. Passing to a subsequence if necessary, we may assume that $P_n\equiv P$ for some $P\in\mathbb{Z}^+$. Note that by Lemma~\ref{lem:factsinq2}(3), \eqref{eq:thesequence} becomes 
\begin{equation}\label{eq:originallimit}
    \h^s\Big( F_{\j_n} \cap \bigcup_{\i\in I_{k_n}: f(K_{\i})\subset F_{\j_n}}f(K_{\i}) \Big) > (1-n^{-1})\h^s(F_{\j_n}).
\end{equation}
For convenience, enumerate $\{\i\in I_{k_n}: f(K_{\i})\subset F_{\j_n}\}=:\{\i_{n,t}\}_{t=1}^P$.

With the notation, for every $n\geq 1$, $f$ sends each $K_{\i_{n,t}}$ into $F_{\j_n}=\psi_{\j_n}(F)$. This naturally generates a blow-up map 
\begin{align*}
    f_n: \bigcup_{t=1}^P \psi_{\j_n}^{-1}(K_{\i_{n,t}}) &\to F \\
    \psi_{\j_n}^{-1}(x) &\mapsto \psi_{\j_n}^{-1}f(x).
\end{align*}
Clearly, $\lip(f_n)\leq \lip(f)=L$ for all $n$. By~\eqref{eq:originallimit},
\begin{align}
    \lim_{n\to\infty} \frac{\h^s\big( F\cap f_n(\bigcup_{t=1}^P \psi_{\j_n}^{-1}(K_{\i_{n,t}})) \big)}{\h^s(F)} &= \lim_{n\to\infty} \frac{\h^s \big( F \cap \bigcup_{t=1}^P \psi_{\j_n}^{-1}f(K_{\i_{n,t}}) \big)}{\h^s(F)} \notag \\
    &= \lim_{n\to\infty} \frac{\h^s\big( F_{\j_n}\cap  \bigcup_{t=1}^P f(K_{\i_{n,t}}) \big) }{\h^s(F_{\j_n})}=1. \label{eq:originallimit2}
\end{align}

For $1\leq t\leq P$, write $\psi_{\j_n}^{-1}(K_{\i_{n,t}}) =: c_{n,t}O_{n,t}K+b_{n,t}$, where $c_{n,t}:=\lambda_{\j_n}^{-1}r_{\i_{n,t}}$. Recall that $\j_n\in J_{k_n}$ and $\i_{n,t}\in I_{k_n}$. By Lemma~\ref{lem:factsinq2}, it is easy to verify that $c_{n,t}\in [ \frac{\underline{r}\Delta_F}{L}, \frac{\Delta_F}{\underline{\lambda} L} ]$ and
\begin{align*}
    \min_{t\neq t'} \dist(\psi_{\j_n}^{-1}(K_{\i_{n,t}}),\psi_{\j_n}^{-1}(K_{\i_{n,t'}})) &= \lambda_{\j_n}^{-1}\min_{t\neq t'} \dist(K_{\i_{n,t}}, K_{\i_{n,t'}}) \\
    &> \lambda_{\j_n}^{-1}\Delta_K\delta^{k_n} \\
    &\geq L^{-1}\Delta_K\Delta_F.
\end{align*}
Applying Lemma~\ref{lem:takinglimits}, we can find a Lipschitz map $g: \bigcup_{t=1}^P (c_tO_tK+b_t) \to F$, where $c_t\in [ \frac{\underline{r}\Delta_F}{L}, \frac{\Delta_F}{\underline{\lambda} L} ]$, $O_t$ is a $d\times d$ orthogonal matrix and $b_t\in\R^d$, such that $\lip(g)\leq L$ and (recalling~\eqref{eq:originallimit2})
\[
    \h^s\Big( g\Big( \bigcup_{t=1}^P (c_tO_tK+b_t) \Big) \Big) \geq \lim_{n\to\infty} \h^s\Big( F\cap f_n\Big( \bigcup_{t=1}^P \psi_{\j_n}^{-1}(K_{\i_{n,t}}) \Big) \Big) = \h^s(F).
\]
Therefore, the image $g( \bigcup_{t=1}^P (c_tO_tK+b_t) )$ is a closed subset of $F$ with full measure. It follows that $g( \bigcup_{t=1}^P (c_tO_tK+b_t) )=F$.

By Baire's theorem, there is some $t$ such that $g(c_tO_tK+b_t)$ contains an interior part of $F$, say $F_{\j}\subset g(c_tO_tK+b_t)$. Then, similarly as in the proof of Lemma~\ref{lem:lipschitzembedding}, it is easy to construct a Lipschitz map sending $c_tO_tK+b_t$ onto $F_{\j}$. Since $c_tO_tK+b_t$ and $F_{\j}$ are nothing but scaled copies of $K$ and $F$, respectively, this establishes Theorem~\ref{thm:main1reduction}.

\section{Lipschitz surjections with measure-preserving properties}

In this section, we fix a Lipschitz embedding $f:K\to F$ with $\h^s(f(K))>0$. By Theorem~\ref{thm:main1reduction}, we may assume $f$ to be surjective. Again write $L:=\lip(f)$ for simplicity.  A key insight, noted in~\cite{CP88, XR08}, asserts that any bi-Lipschitz map between $K$ and $F$ (if there is any) should preserve the $s$-dimensional Hausdorff measure in some local area; see~\cite[Lemma~2.4]{RRW12} for a short proof. While the given map $f$ may fall short of this ideal, we can nevertheless produce an alternative Lipschitz surjection from $K$ to $F$ that enjoys some nice measure linearity.

Let $\mathcal{K}:=\mathcal{K}(P,\gamma,u,v)$ be the collection as in Lemma~\ref{lem:takinglimits} with $[u,v]:=[\frac{\underline{r}\Delta_F}{L}, \frac{\Delta_F}{\underline{\lambda}L}]$, $P:=\lfloor \frac{L^s}{\underline{r}^s\Delta_F^s} \rfloor$ and $\gamma:=L^{-1}\Delta_F\Delta_K$.

\begin{lemma}
    The collection 
    \[
        \{E\in\mathcal{K}: \exists\text{ a Lipschitz} \text{  surjection $g:E\to F$ with $\lip(g)\leq L$} \}
    \]
    is non-empty.
\end{lemma}
\begin{proof}
    Let $c_1:=\min_{\j\in J_1}\frac{\h^s(f^{-1}(F_{\j}))}{\h^s(F_{\j})}$. Since $f(K)=F$, 
    \begin{equation}\label{eq:boundofc1}
        \frac{\h^s(K)}{\h^s(F)} = \sum_{\eta\in J_1} \frac{\h^s(F_{\eta})}{\h^s(F)}\cdot\frac{\h^s(f^{-1}(F_{\eta}))}{\h^s(F_{\eta})} \geq c_1\sum_{\eta\in J_1} \frac{\h^s(F_{\eta})}{\h^s(F)} = c_1.
    \end{equation}
    Pick $\j\in J_1$ with $\frac{\h^s(f^{-1}(F_{\j}))}{\h^s(F_{\j})}=c_1$. By Lemma~\ref{lem:factsinq2}(3), $f^{-1}(F_{\j})$ is a finite union $\bigcup_{\i\in\mathcal{I}}K_{\i}$ for some $\mathcal{I}\subset I_1$. Then we can blow up $f: \bigcup_{\i\in\mathcal{I}}K_{\i}\to F_{\j}$ and get a Lipschitz surjection 
    \begin{align*}
        g: \bigcup_{\i\in\mathcal{I}} \psi_{\j}^{-1}K_{\i} &\to F \\
        \psi_{\j}^{-1}(x) &\mapsto \psi_{\j}^{-1}f(x)
    \end{align*}
    with $\lip(g)\leq\lip(f)=L$.
    
    It remains to check that $\bigcup_{\i\in\mathcal{I}} \psi_{\j}^{-1}(K_{\i})\in\mathcal{K}$. By Lemma~\ref{lem:factsinq2}, we have 
    \[
        \frac{\underline{r}\Delta_F}{L} \leq \lambda_{\j}^{-1}r_{\i} \leq \frac{\Delta_F}{\underline{\lambda}L}, \quad \forall \i\in\mathcal{I}
    \]
    and 
    \begin{align*}
        \dist(\psi_{\j}^{-1}(K_{\i}),\psi_{\j}^{-1}(K_{\i'})) = \lambda_{\j}^{-1}\dist(K_{\i}, K_{\i'}) \geq L^{-1}\Delta_F\Delta_K =\gamma, \quad \forall\i\neq\i'\in\mathcal{I}.
    \end{align*}
	Moreover, since 
	\begin{align*}
		\#\mathcal{I}\cdot (\underline{r}\delta)^s\h^s(K) &\leq \#\mathcal{I}\cdot\min_{\i\in I_1} \h^s(K_{\i}) &&\text{(by Lemma~\ref{lem:factsinq2}(1))}\\
		&\leq \h^s\Big( \bigcup_{\i\in\mathcal{I}}K_{\i} \Big) \\
		&= \h^s(f^{-1}(F_{\j})) \\
		&= c_1\h^s(F_{\j}) \\
		&\leq \frac{\h^s(K)}{\h^s(F)}\cdot\Big( \frac{L\delta}{\Delta_F} \Big)^s\h^s(F), &&\text{(by~\eqref{eq:boundofc1} and Lemma~\ref{lem:factsinq2}(1))}
	\end{align*}
	we have $\#\mathcal{I} \leq \lfloor\frac{L^s}{\underline{r}^s\Delta_F^s}\rfloor = P$. So $\bigcup_{\i\in\mathcal{I}}\psi_{\j}^{-1} (K_{\i})\in\mathcal{K}$. 
   
    We remark that 
    \begin{equation}\label{eq:thefirstboundedmass}
        \frac{\h^s(\bigcup_{\i\in\mathcal{I}}\psi_{\j}^{-1} (K_{\i}))}{\h^s(F)} = \frac{\h^s(\bigcup_{\i\in\mathcal{I}} K_{\i})}{\h^s(F_{\j})} = \frac{\h^s(f^{-1}(F_{\j}))}{\h^s(F_{\j})} = c_1 \leq \frac{\h^s(K)}{\h^s(F)},
    \end{equation}
    which will be used later.
\end{proof}

\begin{corollary}\label{lem:infoflcellstof}
    We can find $\widetilde{K}\in\mathcal{K}$ and a Lipschitz surjection $\widetilde{f}: \widetilde{K}\to F$ with $\lip(\widetilde{f})\leq L$ such that 
    \begin{align*}
        \h^s(\widetilde{K}) = \inf\{ \h^s(E): E\in\mathcal{K}, \exists\text{ a Lipschitz} &\text{  surjection $g:E\to F$ with $\lip(g)\leq L$}\}.
    \end{align*}
    Moreover, $\h^s(\widetilde{K})\leq \h^s(K)$.
\end{corollary}
\begin{proof}
    The above lemma tells us that the infimum is well defined. Then one can take $\{E_n\}_{n=1}^\infty\subset\mathcal{K}$ and a sequence $\{g_n\}_{n=1}^\infty$, where each $g_n:E_n\to F$ is a Lipschitz surjection with $\lip(g_n)\leq L$, such that $\lim_n \h^s(E_n)$ equals the infimum. By Lemma~\ref{lem:takinglimits}, we can find $\widetilde{K}\in\mathcal{K}$ and a Lipschitz map $\widetilde{f}:\widetilde{K}\to F$ with $\lip(\widetilde{f})\leq L$ such that 
    \[
        \h^s(\widetilde{K}) = \lim_{j\to\infty}\h^s(E_{n_j}) \quad\text{and}\quad \h^s(\widetilde{f}(K)) \geq \limsup_{j\to\infty} \h^s(g_{n_j}(E_{n_j})) = \h^s(F)
    \]
    for some subsequence $\{n_j\}$. The former implies that $\h^s(\widetilde{K})$ equals the desired infimum and thus does not exceed $\h^s(K)$ (recall~\eqref{eq:thefirstboundedmass}); the latter implies that $\widetilde{f}$ is onto.
\end{proof}

In the sequel of the paper, let us fix such an optimal $\widetilde{K}\in \mathcal{K}$, say $\widetilde{K}=\bigcup_{t=1}^p (a_tO_tK+b_t)$, and a Lipschitz surjection $\widetilde{f}: \widetilde{K} \to F$ with $\lip(\widetilde{f})\leq L$. Write $\widetilde{c}:=\frac{\h^s(\widetilde{K})}{\h^s(F)}$. Also fix a small constant $0<\delta<\min\{1,a_1,a_2,\ldots,a_p\}$.

We next prove some nice measure-preserving properties of $\widetilde{f}:\widetilde{K}\to F$ enforced by the minimality of $\h^s(\widetilde{K})$. Since the domain of $\widetilde{f}$ is a union of scaled copies of $K$ rather than $K$ itself, it is more convenient to use 
\begin{equation}\label{eq:iprimetn}
    I'_{t,n} := \{\i\in I^*: a_tr_{\i} \leq\delta^n < a_tr_{\i^-}\}, \quad 1\leq t\leq p 
\end{equation}
instead of $I_n$ (defined in~\eqref{eq:originalinjn}).

\begin{lemma}\label{lem:tworeofa}
    For every $n$ and $\j\in J_n$, $\widetilde{f}^{-1}(F_{\j}) = \bigcup_{t=1}^p \bigcup_{\i\in\mathcal{I}_t} (a_tO_tK_{\i}+b_t)$ with $\mathcal{I}_t\subset I'_{t,n}$ for all $1\leq t\leq p$. Moreover, if $\frac{\h^{s}(\widetilde{f}^{-1}(F_{\j}))}{\h^s(F_{\j})}\leq\widetilde{c}$, then $\psi_{\j}^{-1}(\widetilde{f}^{-1}(F_{\j}))\in\mathcal{K}$ and $\frac{\h^{s}(\widetilde{f}^{-1}(F_{\j}))}{\h^s(F_{\j})}=\widetilde{c}$.
\end{lemma}
\begin{proof}
    Fix $n$ and $\j\in J_n$. Firstly, for each $1\leq t\leq p$, if $\i\in I'_{t,n}$ then by Lemma~\ref{lem:factsinq2}(2),
    \begin{align}
        \diam(\widetilde{f}(a_tO_tK_{\i}+b_t)) &\leq L\diam(a_tO_tK_{\i}+b_t) \notag \\
        &= La_tr_{\i} 
        \leq L\delta^n < \min_{\eta\neq\eta'\in J_n} \dist(F_{\eta},F_{\eta'}), \label{eq:fimageinonecell}
    \end{align} 
    that is, $\widetilde{f}(a_tO_tK_{\i}+b_t)$ should be contained in a unique $F_\eta$. Therefore, one can write 
    \begin{equation}\label{eq:expressionoffinv}
        \widetilde{f}^{-1}(F_{\j}) = \bigcup_{t=1}^p \bigcup_{\i\in\mathcal{I}_t} (a_tO_tK_{\i}+b_t),
    \end{equation}
    where $\mathcal{I}_t\subset I'_{t,n}$ for all $1\leq t\leq p$.
    
    Secondly, we show that $\psi_{\j}^{-1}\widetilde{f}^{-1}(F_{\j}) = \bigcup_{t=1}^p \bigcup_{\i\in\mathcal{I}_t} \psi_{\j}^{-1}(a_tO_tK_{\i}+b_t)$ is a member of $\mathcal{K}$ provided that $\frac{\h^{s}(\widetilde{f}^{-1}(F_{\j}))}{\h^s(F_{\j})}\leq\widetilde{c}$. Let $1\leq t,t'\leq p$, $\i\in\mathcal{I}_t$ and $\i'\in\mathcal{I}_{t'}$. By Lemma~\ref{lem:factsinq2}(1), every ratio $\lambda_{\j}^{-1}a_tr_{\i}\in[\frac{\underline{r}\Delta_F}{L},\frac{\Delta_F}{\underline{\lambda}L}]$. Furthermore, if $t\neq t'$ then since $\widetilde{K}\in\mathcal{K}$,
    \begin{align*}
        \dist(\psi_{\j}^{-1}(a_tO_tK_{\i}+b_t), & \psi_{\j}^{-1}(a_{t'}O_{t'}K_{\i'}+b_{t'})) \\
        &> \dist(a_tO_tK_{\i}+b_t, a_{t'}O_{t'}K_{\i'}+b_{t'}) \\
        &\geq \dist(a_tO_tK+b_t, a_{t'}O_{t'}K+b_{t'}) 
        \geq \gamma;
    \end{align*}
    while if $t=t'$ and $\i\neq\i'$ then 
    \begin{align*}
        \dist(\psi_{\j}^{-1}(a_tO_t &K_{\i}+b_t), \psi_{\j}^{-1}(a_{t'}O_{t'}K_{\i'}+b_{t'})) \\
        &= \lambda_{\j}^{-1}a_t\dist(K_{\i}, K_{\i'}) \\
        &\geq \lambda_{\j}^{-1}a_tr_{\i^-}\Delta_K 
        \geq \frac{\Delta_F}{L\delta^n}\cdot \delta^n\cdot\Delta_K = L^{-1}\Delta_F\Delta_K = \gamma
    \end{align*}
    as well.  

    Now assume that $\frac{\h^{s}(\widetilde{f}^{-1}(F_{\j}))}{\h^s(F_{\j})}\leq\widetilde{c}$. Then since $\h^s(\widetilde{K})\leq \h^s(K)$ (see Corollary~\ref{lem:infoflcellstof}) and $\j\in J_n$,
    \[
        \h^s(\widetilde{f}^{-1}(F_{\j})) \leq \widetilde{c}\h^s(F_{\j})= \frac{\h^s(\widetilde{K})}{\h^s(F)}\cdot\lambda_{\j}^s\h^s(F)\leq \lambda_{\j}^s\h^s(K)\leq \frac{(L\delta^n)^s}{\Delta_F^s}\h^s(K).
    \]
    On the other hand, the expression~\eqref{eq:expressionoffinv} immediately implies
    \[
        \h^s(\widetilde{f}^{-1}(F_{\j})) \geq \Big( \sum_{t=1}^p \#\mathcal{I}_t \Big)\cdot (\underline{r}\delta^n)^s\h^s(K). 
    \]
    Thus $\sum_{t=1}^p \#\mathcal{I}_t \leq \lfloor\frac{L^s}{\underline{r}^s\Delta_F^s}\rfloor = P$. Altogether, we have $\psi_{\j}^{-1}\widetilde{f}^{-1}(F_{\j})\in\mathcal{K}$. 

    Furthermore, the inclusion $\widetilde{f}: \widetilde{f}^{-1}(F_{\j})\to F_{\j}$ naturally generates a Lipschitz surjection
    \begin{align*}
        g: \psi_{\j}^{-1}(\widetilde{f}^{-1}(F_{\j})) &\to F \\
        \psi_{\j}^{-1}(x) &\mapsto \psi_{\j}^{-1}\widetilde{f}(x)
    \end{align*}
    with $\lip(g)\leq \lip(\widetilde{f})\leq L$. Since $\psi_{\j}^{-1}\widetilde{f}^{-1}(F_{\j})\in\mathcal{K}$, it follows from our choice of $\widetilde{K}$ that 
    \[
        \h^s(\psi_{\j}^{-1}\widetilde{f}^{-1}(F_{\j})) \geq \h^s(\widetilde{K}) = \widetilde{c}\h^s(F),
    \]
    which implies that $\frac{\h^s(\widetilde{f}^{-1}(F_{\j}))}{\h^s(F_{\j})}=\widetilde{c}$ as desired.
\end{proof}

\begin{proposition}\label{prop:linearityofpreimage}
    For all $n\geq 1$ and $\j\in J_n$, $\frac{\h^s(\widetilde{f}^{-1}(F_{\j}))}{\h^s(F_{\j})} = \widetilde{c}$ and $\psi_{\j}^{-1}\widetilde{f}^{-1}(F_{\j})\in\mathcal{K}$.
\end{proposition}
\begin{proof}
    Let $n\geq 1$. Since $\widetilde{f}:\widetilde{K} \to F$ is onto, 
    \begin{align*}
        \widetilde{c} = \frac{\h^s(\widetilde{K})}{\h^s(F)} = \sum_{\j\in J_n} \frac{\h^s(F_{\j})}{\h^s(F)}\cdot\frac{\h^s(\widetilde{f}^{-1}(F_{\j}))}{\h^s(F_{\j})} = \sum_{\j\in J_n} \lambda_{\j}^s \frac{\h^s(\widetilde{f}^{-1}(F_{\j}))}{\h^s(F_{\j})},
    \end{align*}
    which is a convex combination. The identity should be $\widetilde{c}=\sum_{\j\in J_n}\lambda_{\j}^s\widetilde{c}$. Otherwise, since every term in that combination is positive, there exists $\j\in J_n$ such that $\frac{\h^s(\widetilde{f}^{-1}(F_{\j}))}{\h^s(F_{\j})}<\widetilde{c}$, contradicting Lemma~\ref{lem:tworeofa}. In particular, $\frac{\h^s(\widetilde{f}^{-1}(F_{\j}))}{\h^s(F_{\j})}=\widetilde{c}$ for every $\j\in J_n$, which in turn implies by Lemma~\ref{lem:tworeofa} that $\psi_{\j}^{-1}\widetilde{f}^{-1}(F_{\j})\in\mathcal{K}$.
\end{proof}

\begin{remark}\label{rem:toallborelsubsets}
    It is standard to extend the above measure-preserving property to all Borel subsets of $F$, i.e., $\frac{\h^s(\widetilde{f}^{-1}(B))}{\h^s(B)}=\widetilde{c}$ for all Borel $B\subset F$. This fact will be used later. 
\end{remark}

In other words, the Lipschitz surjection $\widetilde{f}$ is measure-preserving in some rough sense. This immediately yields the following corollary, recovering the aforementioned result of Balka and Keleti (Theorem~\ref{thm:bk24}). 

\begin{corollary}[\cite{BK24}]\label{cor:bk24}
    If $\Phi$ is homogeneous, then $K$ and $F$ are Lipschitz equivalent.
\end{corollary}
\begin{proof}
    Write $r$ to be the common ratio of similitudes in $\Phi$. By Lemma~\ref{lem:logdepimpliesequi}, it suffices to show that $\frac{\log\lambda_j}{\log r}\in\mathbb{Q}$ for $j\in J$. Let $j\in J$ and pick for $n\geq 1$ a positive integer $k_n$ so that 
    \[
        j^{k_n}:=\underbrace{jj\cdots j}_{k_n\text{ terms}}\in J_n.
    \]
    For any $1\leq t\leq p$, let $q_{t,n}$ be the largest positive integer such that $r^{q_{t,n}}\leq \delta^n/a_t$. Since $r_i= r$ for all $i\in I$, $I'_{t,n}=I^{q_{t,n}}$ (recall~\eqref{eq:iprimetn}). Then for $t\neq t'$, we always have $a_tr^{q_{t,n}}\leq \delta^n<a_{t'}r^{q_{t',n}-1}$ and hence $r^{q_{t',n}-q_{t,n}}>a_tr/a_{t'}$. So there is a constant $C\geq 1$ such that 
    \begin{equation}\label{eq:boundednessofqtqt}
    	|q_{t,n}-q_{t',n}| \leq C, \quad \forall 1\leq t,t'\leq p, \forall n\in\mathbb{Z}^+.
    \end{equation} 
    
    By Lemma~\ref{lem:tworeofa} and Proposition~\ref{prop:linearityofpreimage},  
    \begin{equation}\label{eq:finitepossofjknmass}
        \h^s(\widetilde{f}^{-1}(F_{j^{k_n}})) \in \Big\{ \h^s(K)\sum_{t=1}^p c_{n,t}a_t^sr^{q_{t,n}s}: c_{n,t}\in\mathbb{Z}^+\cup\{0\}  \text{ with } \sum_{t=1}^p c_{n,t}\leq P \Big\}.
    \end{equation} 
    Writing
    \[
        \mathcal{B} := \Big\{ \sum_{t=1}^p c_{t}a_t^sr^{\beta_t s}: c_{t}\in\mathbb{Z}^+\cup\{0\}, \beta_t\in\mathbb{Z} \text{ are such that} \sum_{t=1}^p c_{t}\leq P \text{ and} -C\leq \beta_t\leq C \Big\},
    \]
    we see from~\eqref{eq:boundednessofqtqt} and~\eqref{eq:finitepossofjknmass} that
    \[
        \h^s(\widetilde{f}^{-1}(F_{j^{k_n}})) \in \{\h^s(K)r^{q_{1,n}s}b: b\in\mathcal{B}\}, \quad \forall n\geq 1.
    \]
    Since $\mathcal{B}$ is a finite collection, we can find (by the pigeonhole principle) a pair of distinct integers $n,\tilde{n}$ such that 
    \[
        \frac{\h^s(\widetilde{f}^{-1}(F_{j^{k_n}}))}{\h^s(\widetilde{f}^{-1}(F_{j^{k_{\tilde{n}}}}))} = r^{(q_{1,n}-q_{1,\tilde{n}})s}
    \]
    whereas $q_{1,n}\neq q_{1,\tilde{n}}$ and $k_n\neq k_{\tilde{n}}$ (these automatically hold when $|n-\tilde{n}|$ is large enough). By Proposition~\ref{prop:linearityofpreimage}, we also have 
    \[
        \frac{\h^s(\widetilde{f}^{-1}(F_{j^{k_n}}))}{\h^s(\widetilde{f}^{-1}(F_{j^{k_{\tilde{n}}}}))} = \frac{\h^s(F_{j^{k_n}})}{\h^s(F_{j^{k_{\tilde{n}}}})} = \lambda_j^{(k_n-k_{\tilde{n}})s}.
    \]
    Thus $\frac{\log\lambda_j}{\log r}\in\mathbb{Q}$. 
\end{proof}

Unfortunately, the above argument fails without the homogeneity of $\Phi$, and even granting the homogeneity of $\Psi$ seems to offer little help. This reveals a subtle asymmetry: to upgrade a Lipschitz embedding to a Lipschitz isomorphism, assuming the homogeneity on $K$ is markedly easier than assuming that on $F$. It is noteworthy that Proposition~\ref{prop:linearityofpreimage} may relate to a characterization given by Xi~\cite{Xi10} via finding appropriate graph-directed subsystems. However, we do not know whether the measure-preserving property in Proposition~\ref{prop:linearityofpreimage} would directly ensure the condition in Xi's criterion (and thereby conclude the equivalence). 

A natural step forward would be to ask how far the optimal map $\widetilde{f}$ deviates from being bi-Lipschitz. The next lemma shows that $\widetilde{f}$ is almost injective and locally nondegenerate.

\begin{lemma}\label{lem:disjointness}
    Let $n\geq 1$, $\j\in J_n$ and write $\widetilde{f}^{-1}(F_{\j})=\bigcup_{t=1}^p\bigcup_{\i\in\mathcal{I}_t} (a_tO_tK_{\i}+b_t)$ (where $\mathcal{I}_t\subset I'_{t,n}$) as in Lemma~\ref{lem:tworeofa}. Then the following properties hold. 
    \begin{enumerate}
        \item $\h^s(\widetilde{f}(a_tO_tK_{\i}+b_t)) \approx \h^s(F_{\j})$ for all $1\leq t\leq p$ and $\i\in\mathcal{I}_t$;
        \item $\h^s(\widetilde{f}(a_tO_tK_{\i}+b_t)\cap \widetilde{f}(a_{t'}O_{t'}K_{\i'}+b_{t'}))=0$ unless $(t,\i)=(t',\i')$.
    \end{enumerate}
\end{lemma}
\begin{proof}
    The first statement is simple: since $\widetilde{f}(a_tO_tK_{\i}+b_t)\subset F_{\j}$, the $\lesssim$ direction holds trivially; for the inverse direction, one can apply Remark~\ref{rem:toallborelsubsets} to $\widetilde{f}(a_tO_tK_{\i}+b_t)$ and get 
    \begin{align*}
        \h^s(\widetilde{f}(a_tO_tK_{\i}+b_t)) &= \widetilde{c}^{-1}\h^s\big( \widetilde{f}^{-1}(\widetilde{f}(a_tO_tK_{\i}+b_t)) \big) \\
    	&\geq \widetilde{c}^{-1}\h^s(a_tO_tK_{\i}+b_t) \approx \h^s(F_{\j}), 
    \end{align*}
    where the last estimate holds because $\lambda_{\j}\approx a_tr_{\i}$ when $\i\in I'_{t,n}$.
    
    It remains to verify the second statement. Since $\widetilde{f}: \widetilde{K}\to F$ is a surjection,
    \[
        F_{\j}=\widetilde{f}\big(\widetilde{f}^{-1}(F_{\j})\big)=\bigcup_{t=1}^p\bigcup_{\i\in\mathcal{I}_{t}} \widetilde{f}(a_t O_{t}K_{\i}+b_{t}).
    \]
    Let $(t,\i)\neq (t',\i')$. Write $E:=\widetilde{f}(a_tO_tK_{\i}+b_t)\cap \widetilde{f}(a_{t'}O_{t'}K_{\i'}+b_{t'})$ and suppose on the contrary that $\h^s(E)>0$.  Then $E$ is a Borel subset of the dust-like self-similar set $F_{\j}$ with positive $\h^s$-measure. By \cite[Lemma~2.12]{EKM10}, we can find positive integers $m_1<m_2<\cdots$, and $\eta_k\in J_{m_k}$ (with $F_{\eta_k}\subset F_{\j}$) such that
    \begin{equation}\label{eq:limitofecapfetak}
        \h^s(E\cap F_{\eta_k})>\big(1-k^{-1}\big) \h^s(F_{\eta_k}), \quad \forall k\geq 1.
    \end{equation}
    
    For each $k\geq 1$, consider $A_k:=\widetilde{f}^{-1}(F_{\eta_k})$ and let 
    \[
        A_{k,1} := \widetilde{f}^{-1}(F_{\eta_k})\cap(a_tO_tK_{\i}+b_t) \quad\text{and}\quad A_{k,2} := \widetilde{f}^{-1}(F_{\eta_k})\cap(a_{t'}O_{t'}K_{\i'}+b_{t'}).
    \]
    Clearly, $A_{k,1}\cap A_{k,2}=\varnothing$ (since $(t,\i)\neq(t',\i')$) and $A_{k,1}\cup A_{k,2}\subset A_k$. Note that 
    \begin{equation}\label{eq:widef-intersec-equal}
        \widetilde{f}(A_{k,1})=F_{\eta_k}\cap \widetilde{f}(a_tO_tK_{\i}+b_t)
    \end{equation}
    (because $\widetilde{f}$ is onto and thus $\widetilde{f}(\widetilde{f}^{-1}(A)\cap B)=A\cap\widetilde{f}(B)$ for all $A\subset F$ and $B\subset K$). Combining this with $\h^s(E\cap F_{\eta_k})>0$, we see that $A_{k,1}\neq\varnothing$. Similarly, $A_{k,2}\neq\varnothing$.  Moreover, it follows from~\eqref{eq:limitofecapfetak} and~\eqref{eq:widef-intersec-equal} that 
    \begin{equation}\label{eq:limitofecapfetakcor}
    	1\geq \lim_{k\to\infty} \frac{\h^s(\widetilde{f}(A_{k,1}))}{\h^s(F_{\eta_k})} = \lim_{k\to\infty} \frac{\h^s(F_{\eta_k}\cap \widetilde{f}(a_tO_tK_{\i}+b_t))}{\h^s(F_{\eta_k})} \geq \lim_{k\to\infty} \frac{\h^s(F_{\eta_k}\cap E)}{\h^s(F_{\eta_k})} = 1,
    \end{equation}
    that is, \eqref{eq:limitofecapfetakcor} is an equality.
    
    By Lemma~\ref{lem:tworeofa}, we can write $\widetilde{f}^{-1}(F_{\eta_k})=\bigcup_{\ell=1}^p\bigcup_{\omega\in\mathcal{I}_{\ell,m_k}} (a_\ell O_{\ell}K_\omega+b_{\ell})$ for some $\mathcal{I}_{\ell,m_k}\subset I'_{\ell,m_k}$. Since $F_{\eta_k}\subset F_{\j}$, 
    \[
        A_{k,1} = \bigcup_{\omega\in\mathcal{I}_{t,m_k}:K_\omega\subset K_{\i}} (a_{t}O_{t}K_{\omega}+b_t) \quad\text{and}\quad A_{k,2} = \bigcup_{\omega\in\mathcal{I}_{t',m_k}:K_\omega\subset K_{\i'}} (a_{t'}O_{t'}K_{\omega}+b_{t'}).
    \]
    By Proposition~\ref{prop:linearityofpreimage}, $\psi_{\eta_k}^{-1}\widetilde{f}^{-1}(F_{\eta_k})\in\mathcal{K}$. So $\psi_{\eta_k}^{-1}(A_{k,1}),\psi_{\eta_k}^{-1}(A_{k,2})\in\mathcal{K}$ (because they are non-empty sub-unions). Also
    \begin{align*}
    	\h^s(A_{k,2}) &\geq \min_{\omega\in I'_{t',m_k}} \h^s(a_{t'}O_{t'}K_{\omega}+b_{t'}) \\
    	&\geq (\underline{r}\delta^{m_k})^s \h^s(K) && \text{(recall~\eqref{eq:iprimetn})} \\
    	&\geq \frac{\underline{r}^s\Delta_F^s}{L^s}\cdot \lambda_{\eta_k}^s\h^s(K) && \text{(since $\eta_k\in J_{m_k}$)} \\
    	&= \frac{\underline{r}^s\Delta_F^s}{L^s\h^s(F)}\h^s(F_{\eta_k}) \h^s(K). 
    \end{align*}
    Recalling that $A_{k,1}\cup A_{k,2}\subset A_k$ (which is a disjoint union) and $\frac{\h^s(A_k)}{\h^s(F_{\eta_k})}=\widetilde{c}$ (Proposition~\ref{prop:linearityofpreimage}), we have 
    \begin{equation}\label{eq:ak1losesmass}
    	\frac{\h^s(A_{k,1})}{\h^s(F_{\eta_k})} \leq \frac{\h^s(A_k)-\h^s(A_{k,2})}{\h^s(F_{\eta_k})} \leq \widetilde{c}-\frac{\underline{r}^s\Delta_F^s\h^s(K)}{L^s\h^s(F)}<\widetilde{c}.
    \end{equation}

    Similarly as before, the restriction $\widetilde{f}:A_{k,1}\to F_{\eta_k}$ naturally generates a Lipschitz map 
    \begin{align*}
    	g_k: \psi_{\eta_k}^{-1}(A_{k,1}) &\to F \\
    	\psi_{\eta_k}^{-1}(x) &\mapsto \psi_{\eta_k}^{-1}\widetilde{f}(x)
    \end{align*}
    with $\lip(g_k)\leq\lip(\widetilde{f})\leq L$. Since $\psi_{\eta_k}^{-1}(A_{k,1})\in\mathcal{K}$, we can find by Lemma~\ref{lem:takinglimits} a set $B\in \mathcal{K}$ and a Lipschitz map $g:B \to F$ with $\lip(g)\leq L$ such that 
    \[
        \h^s(B) \leq \sup_k\h^s(\psi_{\eta_k}^{-1}(A_{k,1})) = \h^s(F)\cdot\sup_k \frac{\h^s(A_{k,1})}{\h^s(F_{\eta_k})} < \widetilde{c}\h^s(F)
    \]
    (where the last inequality is due to~\eqref{eq:ak1losesmass}) and 
    \begin{align*}
    	\h^s(g(B)) &\geq \limsup_{k\to\infty} \h^s(g_k(\psi_{\eta_k}^{-1}(A_{k,1}))) \\
    	&= \limsup_{k\to\infty} \h^s(\psi_{\eta_k}^{-1}\widetilde{f}(A_{k,1})) &&\text{(by definition)} \\
    	&= \h^s(F) \cdot \limsup_{k\to\infty} \frac{\h^s(\psi_{\eta_k}^{-1}\widetilde{f}(A_{k,1}))}{\h^s(F)} \\
    	&= \h^s(F).  && \text{(by~\eqref{eq:limitofecapfetakcor})}
    \end{align*}
    The latter implies that $g(B)=F$, i.e., $g: B\to F$ is a Lipschitz surjection with $\lip(g)\leq L$. However, it follows from $B\in\mathcal{K}$ and our choice of $\widetilde{K}$ (recall Corollary~\ref{lem:infoflcellstof}) that  
    \[
    \widetilde{c}\h^s(F) > \h^s(B) \geq \h^s(\widetilde{K}) = \widetilde{c}\h^s(F),
    \]
    which is a contradiction. This proves (1).  
\end{proof}

\section{Lipschitz equivalence via the homogeneity of $F$}

In this section, we restrict our attention to the case when $\Psi$ is homogeneous, say $\lambda_j\equiv\lambda$ for $j\in J$, and establish Theorem~\ref{thm:main2}(1). By Lemma~\ref{lem:logdepimpliesequi}, the proof again reduces to showing that $\frac{\log r_i}{\log \lambda}\in\mathbb{Q}$ for $i\in I$. The following key observation will enable us to, in some sense, interchange the roles of $K$ and $F$ in adapting the proof of Corollary~\ref{cor:bk24}.

\begin{proposition}\label{prop:tildefisnottoosca}
    There exists $\alpha\in\Z^+$ such that for all $1\leq t\leq p$, $n\geq 1$ and $\i\in I'_{t,n}$, $\widetilde{f}(a_tO_tK_{\i}+b_t)$ contains at least one cell in $\{F_{\j}: \j\in J_{n+\alpha}\}$.
\end{proposition}

In other words, the Lipschitz surjection $\widetilde{f}$ is not distributed excessively fragmented in $F$. Since the proof of this proposition is rather involved, let us prove Theorem~\ref{thm:main2}(1) upon it first and then check its validation in the next subsection. 

\subsection{From Proposition~\ref{prop:tildefisnottoosca} to Theorem~\ref{thm:main2}(1)}
Assume Proposition~\ref{prop:tildefisnottoosca} holds and let $i\in I$ be arbitrary.

For $n\geq 1$, let $k_n$ be the unique integer such that $i^{k_n}\in I'_{1,n}$. Proposition~\ref{prop:tildefisnottoosca} provides a word $\j_n\in J_{n+\alpha}$ with $F_{\j_n}\subset\widetilde{f}(a_1O_1K_{i^{k_n}}+b_1)$. By Lemma~\ref{lem:tworeofa}, we can write
\begin{equation}\label{eq:muniontosinglecell}
    \widetilde{f}^{-1}(F_{\j_n}) = \bigcup_{t=1}^p\bigcup_{\i\in\mathcal{I}_t} (a_tO_tK_{\i}+b_t),
\end{equation}
where $\mathcal{I}_t\subset I'_{t,n+\alpha}$. 
For $1\leq t\leq p$, if $\mathcal{I}_{t}\not=\emptyset$ and $\i\in\mathcal{I}_{t}$, then letting $\eta$ be the prefix of $\i$ with $\eta\in I'_{t,n}$, we have
\begin{align*}
    \h^s(\widetilde{f}(a_1O_1K_{i^{k_n}}+b_1)\,\cap\, &\widetilde{f}(a_{t}O_{t}K_\eta+b_{t})) \\
    &\geq \h^s(F_{\j_n}\cap\widetilde{f}(a_{t}O_{t}K_{\i}+b_{t})) \\
    &= \h^s(\widetilde{f}(a_{t}O_{t}K_{\i}+b_{t})) && \text{(by~\eqref{eq:muniontosinglecell})} \\
    &>0 && \text{(by Lemma~\ref{lem:disjointness}(1))}.
\end{align*}
By Lemma~\ref{lem:disjointness}(2), this is impossible unless $t=1$ and $\eta=i^{k_n}$. So~\eqref{eq:muniontosinglecell} should be 
\begin{equation}\label{eq:massoffjninver}
    \widetilde{f}^{-1}(F_{\j_n}) = \bigcup_{\i\in\mathcal{I}_1:K_{\i}\subset K_{i^{k_n}}} (a_1O_1K_{\i}+b_1).
\end{equation}

Note that $\{\i\in\mathcal{I}_1:K_{\i}\subset K_{i^{k_n}}\}$ can be written as $\{i^{k_n}\omega:\omega\in\Gamma_n\}$ for some $\Gamma_n\subset I^*$. Since $i^{k_n}\in I'_{1,n}$ and $\mathcal{I}_1\subset I'_{1,n+\alpha}$, $a_1r_i^{k_n}\approx\delta^n$ and $a_1r_{i}^{k_n}r_\omega\approx\delta^{n+\alpha}$ for all $\omega\in\Gamma_n$. So there exists an integer $q$ independent of $n$ such that $\Gamma_n\subset\bigcup_{k=1}^q I^k$ for all $n$. By the pigeonhole principle, we may choose $n_1<n_2$ such that $\Gamma_{n_1}=\Gamma_{n_2}$. Since $\Psi$ is homogeneous, there are integers $z_1,z_2$ such that $J_{n_1+\alpha}=J^{z_1}$ and $J_{n_2+\alpha}=J^{z_2}$. Recalling Proposition~\ref{prop:linearityofpreimage} and~\eqref{eq:massoffjninver}, we have 
\[
    \widetilde{c} = \frac{\h^s(\widetilde{f}^{-1}(F_{\j_{n_\ell}}))}{\h^s(F_{\j_{n_\ell}})} = \frac{\sum_{\omega\in\Gamma_{n_\ell}} \h^s(a_1O_1K_{i^{k_{n_\ell}}\omega}+b_1)}{\h^s(F_{\j_{n_\ell}})} = \frac{a_1^sr_i^{k_{n_\ell}s}\sum_{\omega\in\Gamma_{n_\ell}} \h^s(K_\omega)}{\lambda^{z_\ell s}\h^s(F)}
\]
for $\ell=1,2$. Since $\Gamma_{n_1}=\Gamma_{n_2}$, we have $r_i^{k_{n_1}-k_{n_2}}=\lambda^{z_1-z_2}$, i.e., $\frac{\log r_i}{\log\lambda}\in\mathbb{Q}$ (choosing $n_1,n_2$ with a large difference, we can require that $z_1\neq z_2$). Since $i$ is arbitrary, Theorem~\ref{thm:main2} is proved.

\subsection{Proof of Proposition~\ref{prop:tildefisnottoosca}}

The purpose of this subsection is to prove Proposition~\ref{prop:tildefisnottoosca} by contradiction. The following landmark result of Szemer\'edi will be useful when establishing a denseness property later in our argument. 

\begin{lemma}[\cite{Sze75}]
    Let $\gamma>0$. For every $k\geq 1$, there exists a threshold number $n(k)$ such that for $n \geq n(k)$, every subset of $\{1, 2, \ldots, n\}$ with cardinality larger than $\gamma n$ contains a $k$-term arithmetic progression.
\end{lemma}

Since $\Psi$ is homogeneous, for each $n$ we can assign a unique integer $\tau(n)$ such that $J_n=J^{\tau(n)}$. For $1\leq t\leq p$ and $n\geq 1$, let
\begin{equation}\label{eq:defofptn}
   q(t,n) := \max\{q\geq 1: \exists \i\in I'_{t,n} \text{ such that $F_{\j}\nsubseteq\widetilde{f}(a_tO_tK_{\i}+b_t)$ for all $\j\in J^{\tau(n)+q}$}\}.
\end{equation}
Suppose on the contrary that Proposition~\ref{prop:tildefisnottoosca} fails. Then $\max_{1\leq t\leq p}\sup_n q(t,n)=\infty$. Without loss of generality, assume that $\sup_{n}q(1,n)=\infty$.

Let $\gamma\in(0,1)$ be a fixed constant (will be specified later). For $k\geq (\# J)^4$, by Szemer\'edi's theorem, there is a large integer $n_k$ such that any subset of $\{1,\ldots,(\#J)^{q(1,n_k)}\}$ with cardinality $\geq \gamma(\#J)^{q(1,n_k)}$ contains a $k$-term arithmetic progression. For notational simplicity, in the rest of the proof we write $q_k:=q(1,n_k)$ and enumerate $J^{q_k}$ as $\{\eta_{k,t}: 1\leq t\leq (\# J)^{q_k}\}$ according to the lexicographical order, that is, 
\[
    \eta_{k,1} = \underbrace{1\cdots 1}_{q_k \text{ terms}},\, \eta_{k,2} = \underbrace{1\cdots 1}_{q_{k}-1 \text{ terms}}\!\!\!\!2,\, \ldots,\, \eta_{k,(\# J)^{q_k}} = \underbrace{(\# J)\cdots(\# J)}_{q_k \text{ terms}}.
\]
It is helpful to note that for any $\j\in J^*$ with $|\j|\leq q_k$, 
\begin{equation}\label{eq:standardsubtree}
    \{t: F_{\eta_{k,t}}\subset F_{\j}\} = \{z(\#J)^{q_k-|\j|}+1, z(\#J)^{q_k-|\j|}+2, \ldots, (z+1)(\#J)^{q_k-|\j|}\}
\end{equation}
for some integer $z\geq 0$, i.e., a $(\#J)$-adic interval of level-$(q_k-|\j|)$.

By~\eqref{eq:defofptn}, one can find a word $\i_k\in I'_{1,n_k}$ with 
\begin{equation}\label{eq:notcontainanyword}
    F_{\j}\nsubseteq\widetilde{f}(a_1O_1K_{\i_k}+b_1), \quad \forall\j\in J^{\tau(n_k)+q_k}.
\end{equation}
On the other hand, similarly as in~\eqref{eq:fimageinonecell}, one can deduce that $\widetilde{f}(a_1O_1K_{\i_k}+b_1)\subset F_{\j_k}$ for some $\j_k\in J_{n_k}=J^{\tau(n_k)}$. Since 
\[
     F_{\j_k} = \bigcup_{\theta\in J^{q_k}} F_{\j_k\theta}=\bigcup_{1\leq t\leq (\#J)^{q_k}} F_{\j_k\eta_{k,t}},
\]
we have
\begin{align}
    \widetilde{f}(a_1O_1K_{\i_k}+b_1)&= \widetilde{f}(a_1O_1K_{\i_k}+b_1) \cap \bigcup_{1\leq t\leq (\#J)^{q_k}} F_{\j_k\eta_{k,t}} \notag \\
    &= \widetilde{f}(a_1O_1K_{\i_k}+b_1) \cap \bigcup_{t\in T_k} F_{\j_k\eta_{k,t}}, \label{eq:imagesupset}
\end{align} 
where $T_k:=\{1\leq t\leq (\#J)^{q_{k}}: \widetilde{f}(a_1O_1K_{\i_k}+b_1) \cap F_{\j_k\eta_{k,t}}\neq\varnothing \}$. Observe that
\begin{align*}
    \delta^{n_ks} &\lesssim \h^s(\widetilde{f}(a_1O_1K_{\i_k}+b_1))  &&\text{(since $\i_k\in I'_{1, n_k}$)} \\
    &\leq \h^s\Big( \bigcup_{t\in T_k} F_{\j_k\eta_{k,t}} \Big)  &&\text{(by~\eqref{eq:imagesupset})} \\
    &= \sum_{t\in T_k} \h^s(F_{\j_k\eta_{k,t}}) \\
    &= \sum_{t\in T_k} \lambda_{\j_k}^s\h^s(F_{\eta_{k,t}}) \\
    &\lesssim (\# T_k) \cdot \delta^{n_ks} \cdot \lambda^{q_ks}\h^s(F)  &&\text{(since $\j_k\in J_{n_k}$ and $\eta_{k,t}\in J^{q_k}$)} 
\end{align*}
and hence $\#T_k\gtrsim \lambda^{-q_ks}=(\#J)^{q_k}$. In summary, although every cell in $\{F_{\j_k\theta}:\theta\in J^{q_k}\}$ is not entirely contained in $\widetilde{f}(a_1O_1K_{\i_k}+b_1)$, the latter does intersect a large proportion of these cells.

Picking $0<\gamma<1$ properly at the beginning, we have $\#T_k\geq \gamma(\#J)^{q_k}$ for all $k$. As a consequence, $T_k$ should contain a $k$-term arithmetic progression, say $AP_k$. Pick $\theta_k\in J^*$ such that $F_{\theta_k}$ is a largest cell contained entirely in $\bigcup_{t=\min AP_k}^{\max AP_k} F_{\eta_{k,t}}$  (so $|\theta_k|\leq q_k$) and set 
\begin{equation}\label{eq:defofpprimek}
    AP'_k := \{t\in AP_k: F_{\eta_{k,t}}\subset F_{\theta_k}\},
\end{equation}
which is a subprogression of $AP_k$. By~\eqref{eq:standardsubtree}, $\{1\leq t\leq (\#J)^{q_k}: F_{\eta_{k,t}}\subset F_{\theta_k}\}$ is simply a $\#J$-adic interval of the form
\begin{equation}\label{eq:alltcontinthetak}
    \{z_k(\#J)^{q_k-|\theta_k|}+1,z_k(\#J)^{q_k-|\theta_k|}+2,\ldots,(z_k+1)(\#J)^{q_k-|\theta_k|}\} 
\end{equation}
for some $z_k\geq 0$. Write $d_k$ to be the common difference of $AP_k$. Clearly,
\begin{equation}\label{eq:cardofpprimek}
	\# AP'_k \geq \frac{\#\{1\leq t\leq (\#J)^{q_k}: F_{\eta_{k,t}}\subset F_{\theta_k}\}}{d_k}-1 = \frac{(\#J)^{q_k-|\theta_k|}}{d_k}-1.
\end{equation}

\begin{lemma}\label{lem:APcardEst}
	$\#AP'_k\geq (\#J)^{-3}k$.
\end{lemma} 
\begin{proof}
	Write $\ell_k:=\lfloor \frac{\log(k-1)d_k}{\log\#J}\rfloor$. Then
	\[
  	   \max AP_k - \min AP_k  =  (\# AP_k-1)d_k = (k-1)d_k\geq (\# J)^{\ell_k} \geq 2\cdot(\#J)^{\ell_k-1}.
	\]
	Therefore, the interval $[\min AP_k, \max AP_k]$ must contain a $\# J$-adic interval of the form 
	\begin{equation*}
		\big\{z(\#J)^{\ell_k-1}+1, \ldots, (z+1)(\#J)^{\ell_k-1}\big\}
	\end{equation*}
	for some $z\geq 0$. By~\eqref{eq:standardsubtree}, this set corresponds to $\{t: F_{\eta_{k,t}}\subset F_{\j}\}$ for some $\j\in J^{q_k-(\ell_k-1)}$. In particular, $F_{\j}\subset \bigcup_{t=\min AP_k}^{\max AP_k} F_{\eta_{k,t}}$. 
	Recalling the choice of $\theta_k$ (above~\eqref{eq:defofpprimek}), we have $|\theta_k|\leq |\j|=q_k-(\ell_k-1)$ and it follows from~\eqref{eq:cardofpprimek} that 
	\begin{align*}
		\# AP'_k & \geq  \frac{(\# J)^{\ell_k-1}}{d_k}-1  \\
		&= (\#J)^{-1}\cdot(\#J)^{\lfloor \frac{\log(k-1)d_k}{\log\#J}\rfloor}\cdot d_k^{-1} - 1 \\
		&\geq (\#J)^{-2}(k-1)d_k\cdot d_k^{-1}-1 \\
		&> k(\#J)^{-2} - 2,
	\end{align*} 
	which is at least $k(\#J)^{-3}$ because $k\geq (\#J)^4$.
\end{proof}

To derive a contradiction, write 
\begin{equation}\label{eq:secondak1}
    A_{k,1} := \widetilde{f}^{-1}(F_{\j_k\theta_k}) \cap (a_1O_1K_{\i_k}+b_1).
\end{equation}

\begin{lemma}\label{lem:lastlem}
    The compact set $A_{k,1}$ satisfies the following two properties.
    \begin{enumerate}
        \item $\psi_{\j_k\theta_k}^{-1}\widetilde{f}(A_{k,1})$ converges to $F$ in the Hausdorff distance as $k\to\infty$.
        \item There exists $0<c_*<1$ such that $\h^s(A_{k,1})\leq (1-c_*)\h^s(\widetilde{f}^{-1}(F_{\j_k\theta_k}))$ for all $k$. 
    \end{enumerate}
\end{lemma}
\begin{proof}
	For (1), write $\widetilde{q}_k:=q_k-|\theta_k|-\lfloor\frac{\log k}{\log\# J}\rfloor+4$. Since $k\geq (\#J)^4$, Lemma~\ref{lem:APcardEst} implies
	\[
	   \#AP_k'-1\geq (\#J)^{-1}\cdot\#AP_k'\geq (\#J)^{-4}k\geq (\#J)^{\lfloor\frac{\log k}{\log\# J}\rfloor-4}
	\]
	and hence
	\begin{align*}
		d_k &\leq \frac{\#\{1\leq t\leq (\# J)^{q_k}: F_{\eta_{k,t}}\subset F_{\theta_k}\}}{\# AP'_k-1} &&\text{(by~\eqref{eq:defofpprimek})} \\
		&= \frac{(\# J)^{q_k-|\theta_k|}}{\# AP'_k-1} &&\text{(by~\eqref{eq:alltcontinthetak})}\\
		&\leq (\# J)^{ q_k-|\theta_k|-\lfloor\frac{\log k}{\log\# J}\rfloor+4} = (\#J)^{\tilde{q}_k}.
	\end{align*}
	So $AP_k$ should intersect every block of $(\#J)^{\widetilde{q}_k}$ consecutive integers in $[\min AP_k, \max AP_k]$ (and in particular, every $\# J$-adic interval of level-$(q_k-\widetilde{q}_k)$ in $[\min AP_k, \max AP_k]$). Since $AP_k\subset T_k$, by the definition of $T_k$ and~\eqref{eq:standardsubtree}, $\widetilde{f}(a_1O_1K_{\i_k}+b_1)$ (and hence $\widetilde{f}(A_{k,1})$) should intersect every level-$(|\j_k|+q_k-\widetilde{q}_k)$ cell contained in $F_{\j_k\theta_k}$. From this, we deduce that 
	\begin{align*}
		\hdist(\widetilde{f}(A_{k,1}), F_{\j_k\theta_k}) \leq \lambda^{|\j_k|+q_k-\widetilde{q}_k} = \lambda_{\j_k}\cdot\lambda^{ |\theta_k|+\lfloor\frac{\log k}{\log\# J}\rfloor-4} = \lambda_{\j_k\theta_k}\cdot \lambda^{ \lfloor\frac{\log k}{\log\# J}\rfloor-4},
	\end{align*}
	where $\hdist$ denotes the Hausdorff distance. Then 
	\[
	    \hdist(\psi^{-1}_{\j_k\theta_k}\widetilde{f}(A_{k,1}), F) = \lambda_{\j_k\theta_k}^{-1}\hdist(\widetilde{f}(A_{k,1}), F_{\j_k\theta_k}) = \lambda^{ \lfloor\frac{\log k}{\log\# J}\rfloor-4},
	\]
	which vanishes as $k\to\infty$. This proves (1).
	
	For (2), note that there are $1\leq t\leq p$, $\i'_k\in I'_{t,n_k}$ with $(t,\i'_k)\neq(1,\i_k)$ such that 
	\begin{equation*}
	    \widetilde{f}(a_tO_tK_{\i'_k}+b_t) \cap F_{\j_k\theta_k} \neq\varnothing;
	\end{equation*}
    otherwise, $F_{\j_k\theta_k}\subset\widetilde{f}(a_1O_1K_{\i_k}+b_1)$, which contradicts~\eqref{eq:notcontainanyword} because $|\theta_k|\leq q_k$. Pick any point $x\in a_tO_tK_{\i'_k}+b_t$ with $\widetilde{f}(x)\in F_{\j_k\theta_k}$. Similarly as in~\eqref{eq:fimageinonecell}, we can find a word $\omega_k\in I^*$ such that $x\in a_tO_tK_{\i'_k\omega_k}+b_t$, $r_{\i'_k\omega_k}\approx \lambda_{\j_k\theta_k}$ and $\widetilde{f}(a_tO_tK_{\i'_k\omega_k}+b_t) \subset F_{\j_k\theta_k}$. Note that 
	\begin{equation}\label{eq:massofikomegak}
		\h^s(a_tO_tK_{\i'_k\omega_k}+b_t) \approx  \h^s(F_{\j_{k}\theta_k}) = \widetilde{c}^{-1}\h^s(\widetilde{f}^{-1}(F_{\j_k\theta_k})),
	\end{equation}
	where the last equality comes from Proposition~\ref{prop:linearityofpreimage}. Then
	\begin{align*}
		\h^s(A_{k,1}) &= \h^s(\widetilde{f}^{-1}(F_{\j_k\theta_k}) \cap (a_1O_1K_{\i_k}+b_1))  &&\text{(by~\eqref{eq:secondak1})} \\
		&\leq \h^s(\widetilde{f}^{-1}(F_{\j_k\theta_k})) - \h^s(a_tO_tK_{\i'_k\omega_k}+b_t) \\
		&\leq (1-c_*)\h^s(\widetilde{f}^{-1}(F_{\j_k\theta_k}))  &&\text{(by~\eqref{eq:massofikomegak})}
	\end{align*}
	for some $0<c_*<1$ independent of $k$, which completes the proof.
\end{proof}

Now we are ready to obtain a contradiction. Again, each Lipschitz inclusion $\widetilde{f}: A_{k,1} \to F_{\j_k\theta_k}$ naturally generates a blow-up map
\begin{align*}
    g_k: \psi_{\j_k\theta_k}^{-1}(A_{k,1}) &\to F \\
    \psi_{\j_k\theta_k}^{-1}(x) &\mapsto \psi_{\j_k\theta_k}^{-1}\widetilde{f}(x)
\end{align*}
with $\lip(g_k)\leq \lip(\widetilde{f})\leq L$. Similar to the fourth paragraph of the proof of Lemma~\ref{lem:disjointness}, we have $\psi_{\j_k\theta_k}^{-1}(A_{k,1})\in\mathcal{K}$ for all $k$. Then, combining Lemma~\ref{lem:takinglimits}, Remark~\ref{rem:eventuallydense} and Lemma~\ref{lem:lastlem}(1), we obtain a Lipschitz surjection $g: E\to F$ for some $E\in\mathcal{K}$ such that $\lip(g)\leq L$ and 
\begin{align*}
    \h^s(E) &\leq \sup_k \h^s(\psi_{\j_k\theta_k}^{-1}(A_{k,1})) \\ 
    &= \h^s(F)\cdot\sup_k \frac{\h^s(A_{k,1})}{\h^s(F_{\j_k\theta_k})} \\
    &\leq \h^s(F)\cdot (1-c_*)\frac{\h^s(\widetilde{f}^{-1}(F_{\j_k\theta_k}))}{\h^s(F_{\j_k\theta_k})} &&\text{(by Lemma~\ref{lem:lastlem}(2))} \\
    &< \widetilde{c}\h^s(F) &&\text{(by Proposition~\ref{prop:linearityofpreimage})} \\
    &= \h^s(\widetilde{K}),
\end{align*} 
which contradicts our choice of $\widetilde{K}$ (recall Lemma~\ref{lem:infoflcellstof}).

\section{A counterexample in the inhomogeneous case}

We finally present a concrete example supporting Theorem~\ref{thm:main2}(2). Following~\cite{XX21}, we say that $\Phi$ is \emph{commensurable} if $\frac{\log r_i}{\log r_{i'}}\in\mathbb{Q}$ for all $i,i'\in I$ and call $r$ the \emph{ratio root}, where $r$ is the unique number in $(0,1)$ such that the multiplicative group generated by $\{r_i:i\in I\}$ equals $r^{\Z}$. Let $\mu$ be the normalized measure of $\h^s$ restricted on $K$. In particular, $\mu(K_{\i})=r_{\i}^s$ for all $\i\in I^*$. 

To state a mass decomposition lemma in~\cite{XX21}, let $\Z[r^s]$ be the ring generated by $r^s$ and the integer set $\Z$. With a slight abuse of notation, here we pick $\delta=r$ and let
\begin{equation}\label{eq:InJnDefLastSection}
    \left\{\begin{array}{l} I_n:= \{\i\in I^*: r_{\i}\leq \delta^n < r_{\i^-}\}, \\ J_n:= \{\j\in J^*: \lambda_{\j} \leq \delta^n < \lambda_{\j^-}\}. \end{array}\right.
\end{equation}

\begin{lemma}[{\cite[Lemma 3.4]{XX21}}]\label{lem:xixiong21}
    Assume that $\Phi$ is commensurable and $r$ is the associated ratio root. Let $\Omega\subset\Z[r^s]\cap(0,+\infty)$ be a finite set. For $k\geq 1$, write $\mathcal{B}_k$ be the collection of all finite union of cells in $\{K_\omega: \omega\in I_k\}$. Then there exists some $c\geq 1$ satisfying the following property: for any $k\geq 1$ and any set $E\in\mathcal{B}_k$, if $\mu(E)=b_1+\cdots+b_n$ for some $b_1,\ldots,b_n\in r^{sk}\Omega$, then we can find a partition $E=\bigcup_{t=1}^n E_t$ such that every $E_t$ is a member in $\mathcal{B}_{k+c}$ with $\mu(E_t)=b_t$.
\end{lemma}

The desired example is as follows.

\begin{example}\label{exa:lastcountexa}
    Consider~\cite[Example 1.1]{XX21}, where $\Phi$ has ratios $r_1=r_2=3^{-1}$, $r_3=r_4=3^{-2}$ and $\Psi$ has ratios $\lambda_1=\cdots=\lambda_{20}=3^{-3}$, $\lambda_{21}=\cdots=\lambda_{28}=3^{-6}$. As stated there, $K$ and $F$ share an identical Hausdorff dimension $s$ but fail to be Lipschitz equivalent. However, by employing the above lemma, it is not hard to construct a Lipschitz surjection from $K$ to $F$. This process can be viewed as a hierarchical mass distribution operation.

    For convenience, write $x:=3^{-s}$ and let $\nu$ be the normalized $\h^s$-measure on $F$. The dimension formula implies that $2x^2+2x=1$ and $8x^6+20x^3=1$. So in $\{F_j: j\in J\}$ there are $8$ cells of $\nu$-mass $x^6$ and $20$ cells of $\nu$-mass $x^3$. Since 
    \[
        (2x^2+2x)^3 = 8x^6+8x^3+24x^5+24x^4,
    \]
    in $\{K_{\i}: \i\in I^3\}$ there are $8$ cells of $\mu$-mass $x^6$, $8$ cells of $\mu$-mass $x^3$, $24$ cells of $\mu$-mass $x^5$ and $24$ cells of $\mu$-mass $x^4$. 
    Noting that $2x^5+2x^4=x^3(2x^2+2x)=x^3$, we can perform the following decomposition based on the $\mu$-mass of cells in $\{K_{\i}:\i\in I^3\}$: 
    \[
        \underbrace{\{x^6\},\ldots\{x^6\}}_{8 \text{ terms}}, \underbrace{\{x^3\},\ldots,\{x^3\}}_{8 \text{ terms}}, \underbrace{\{x^5,x^5,x^4,x^4\},\ldots,\{x^5,x^5,x^4,x^4\}}_{12 \text{ terms}}.
    \] 
    Thus there exist $8$ subgroups with $\mu$-mass $x^6$ and $20$ subgroups with $\mu$-mass $x^3$. This induces a natural partition of $I^3$ that admits a ``mass-preserving'' bijection to $J$. More precisely, we can decompose $I^3=\bigcup_{j\in J} \tilde{I}_{j}$ such that $\sum_{\i\in \tilde{I}_{j}} r_{\i}^s = \lambda_j^s$ for each $j\in J$.
    
    In this example, the ratio root $r$ equals $3^{-1}$, so $\delta=3^{-1}$ and $J_1=J$. Write 
    \[
        \Omega := \{1,3^{-s},3^{-2s},\ldots,3^{-5s}\}.
    \]
    For such parameters $r$ and $\Omega$, there exists a large integer $c$ as specified in Lemma~\ref{lem:xixiong21}.
    We claim that for every $k\geq 1$, there exists a partition $I_{1+kc}=\bigcup_{\j\in J_{1+(k-1)c}} \tilde{I}_{\j}$ such that $\sum_{\i\in\tilde{I}_{\j}}r_{\i}^s = \lambda_{\j}^s$ for each $\j\in J_{1+(k-1)c}$, where $I_n$ and $J_n$ are as defined in~\eqref{eq:InJnDefLastSection}.
    
    The claim is proved by induction. Clearly, picking $c$ large enough initially, every cell in $\{K_{\i}:\i\in I^3\}$ can be expressed as a finite union of cells in $\{K_\omega: \omega\in I_{1+c}\}$. Since $J_1=J$, we can transform the previous partition of $I^3$ into another partition $I_{1+c}=\bigcup_{\j\in J_1} \tilde{I}_{\j}$ with $\sum_{\i\in\tilde{I}_{\j}}r_{\i}^s = \lambda_{\j}^s$ for each $\j\in J_1$. This settles the case $k=1$.

    Suppose we have established the claim for some $k\geq 1$. By definition, $\lambda_{\eta}\leq 3^{-(1+kc)}<\lambda_{\eta^-}$ for each $\eta\in J_{1+kc}$. Since each map in $\Psi$ has ratio either $3^{-3}$ or $3^{-6}$, we have 
    \[
        \lambda_\eta \in \{3^{-(1+kc)},3^{-(1+kc)-1},\ldots,3^{-(1+kc)-5}\}, \quad \forall \eta\in J_{1+kc}
    \]
    and thus 
    \begin{equation}\label{eq:etainsomega}
        \lambda_\eta^s \in 3^{-(1+kc)s}\{1,3^{-s},\ldots,3^{-5s}\} = 3^{-(1+kc)s}\Omega = r^{(1+kc)s}\Omega, \quad \forall \eta\in J_{1+kc}.
    \end{equation}
    Let $\j\in J_{1+(k-1)c}$. By the induction hypothesis, $\j$ corresponds to some $\tilde{I}_{\j}\subset I_{1+kc}$ with
    \[
        \sum_{\i\in \tilde{I}_{\j}} r_{\i}^s = \lambda_{\j}^s = \sum_{\eta\in J_{1+kc}: \j\prec\eta} \lambda_\eta^s,
    \]
    where $\j\prec\eta$ means that $\j$ is a prefix of $\eta$. In other words, writing $E_{\j}:=\bigcup_{\i\in\tilde{I}_{\j}}K_{\i}$, $E_{\j}$ is simply an element in $\mathcal{B}_{1+kc}$ (as in Lemma~\ref{lem:lastlem}) with  $\mu(E_{\j})=\sum_{\eta\in J_{1+kc}: \j\prec\eta} \lambda_\eta^s$, where each $\lambda_\eta^s$ is a member of $r^{(1+kc)s}\Omega$ (recall~\eqref{eq:etainsomega}). It then follows from Lemma~\ref{lem:xixiong21} that there exists a partition $E_{\j}=\bigcup_{\eta\in J_{1+kc}:\j\prec\eta} E_\eta$, where every $E_\eta$ is a member of $\mathcal{B}_{1+(k+1)c}$ with $\mu(E_\eta)=\lambda_\eta^s$. By definition, there is some $\tilde{I}_\eta\subset I_{1+(k+1)c}$ such that $E_\eta=\bigcup_{\xi\in\tilde{I}_\eta} K_\xi$. We conclude that 
    \begin{equation}\label{eq:finalpartition}
        \bigcup_{\i\in\tilde{I}_{\j}} K_{\i} = E_{\j} = \bigcup_{\eta\in J_{1+kc}: \j\prec\eta} E_\eta = \bigcup_{\eta\in J_{1+kc}:\j\prec\eta} \bigcup_{\xi\in\tilde{I}_\eta}K_\xi,
    \end{equation} 
    where each equality represents a partition. 
    It follows that 
    \begin{equation}\label{eq:6.4}
        \bigcup_{\j\in J_{1+(k-1)c}}\bigcup_{\i\in \tilde{I}_{\j}} K_{\i}  
        = \bigcup_{\j\in J_{1+(k-1)c}}\bigcup_{\eta\in J_{1+kc}:\j\prec\eta} \bigcup_{\xi\in \tilde{I}_\eta} K_\xi 
        = \bigcup_{\eta\in J_{1+kc}}\bigcup_{\xi\in\tilde{I}_\eta} K_\xi,  
    \end{equation}
    because $J_{1+kc}=\bigcup_{\j\in J_{1+(k-1)c}}\{\eta\in J_{1+kc}: \j\prec\eta\}$. On the other hand, since $I_{1+kc}=\bigcup_{\j\in J_{1+(k-1)c}}\tilde{I}_{\j}$ (the induction hypothesis), 
    \begin{align*}
    	\bigcup_{\i\in I_{1+kc}} K_{\i} 	= \bigcup_{\j\in J_{1+(k-1)c}}\bigcup_{\i\in \tilde{I}_{\j}} K_{\i}. 
    \end{align*}
    Combining this with~\eqref{eq:6.4}, we have
    \begin{align*}
    	\bigcup_{\xi\in I_{1+(k+1)c}} K_\xi = K = \bigcup_{\i\in I_{1+kc}} K_{\i} = \bigcup_{\eta\in J_{1+kc}}\bigcup_{\xi\in\tilde{I}_\eta} K_\xi,  
    \end{align*} 
    which in turn induces a partition (since each $\tilde{I}_\eta\subset I_{1+(k+1)c}$)
    \[
        I_{1+(k+1)c} = \bigcup_{\eta\in J_{1+kc}} \tilde{I}_\eta
    \]
    such that $\sum_{\xi\in \tilde{I}_{\eta}}r_{\xi}^s=\mu(E_\eta)=\lambda_\eta^s$ for all $\eta\in J_{1+kc}$. This completes the induction.

    The claim naturally induces a Lipschitz surjection from $K$ to $F$ as follows. For $x\in K$ and $k\geq 1$, let $\i_k$ be the unique word in $I_{1+kc}$ such that $x\in K_{\i_k}$. Note that $\i_1\prec\i_2\prec\cdots$ and $\{x\}=\bigcap_k K_{\i_k}$. By our claim, there exists $\j_k\in J_{1+(k-1)c}$ such that $\i_k\in \tilde{I}_{\j_k}$. Moreover,
    \begin{equation}\label{eq:kik1inieta}
        K_{\i_{k+1}}\subset K_{\i_k} \subset \bigcup_{\omega\in\tilde{I}_{\j_k}} K_\omega = \bigcup_{\eta\in J_{1+kc}: \j_k\prec\eta}  \bigcup_{\xi\in\tilde{I}_\eta} K_\xi, 
    \end{equation}
    where the last equality follows from~\eqref{eq:finalpartition}. Since $\i_{k+1}\in I_{1+(k+1)c}$ and $\bigcup_{\eta\in J_{1+kc}:\j_k\prec\eta}\tilde{I}_\eta\subset I_{1+(k+1)c}$, \eqref{eq:kik1inieta} implies that
    \begin{equation*}
        \i_{k+1}\in\bigcup_{\eta\in J_{1+kc}:\j_k\prec\eta}\tilde{I}_\eta.
    \end{equation*}
    Combining this with our choice of $\j_{k+1}$ (i.e., the word in $J_{1+kc}$ such that $\i_{k+1}\in \widetilde{I}_{{\j}_{k+1}}$), we see that $\j_{k+1}\in \{\eta\in J_{1+kc}:\j_k\prec \eta\}$. 
    In particular, $\j_k\prec\j_{k+1}$. So $\j_1\prec\j_2\prec\cdots$ as well and we set $g(x)$ to be the single point in $\bigcap_{k} F_{\j_k}$. 
    
    To see the surjectivity of $g$, let $y\in F$ be arbitrary. For each $k\geq 1$, let $\xi_k$ be the unique word in $\widetilde{J}_{1+(k-1)c}$ such that $y\in F_{\xi_k}$. From the previous induction argument, it follows that $B_k:=\bigcup_{\i\in \widetilde{I}_{\xi_k}}K_{\i}$ is non-empty, compact and nested decreasing in $k$ (see~\eqref{eq:finalpartition}). Moreover, $g$ maps the non-empty intersection $\bigcap_{k=1}^\infty B_k$ to the singleton $\{y\}$. So $g(K)=F$.

    It remains to prove that $g$ is Lipschitz.
    Let $x,x'\in K$ be distinct and let $k_*$ be the maximal integer such that $x,x'\in K_{\omega_*}$ for a common $\omega_*\in I_{1+k_*c}$. So $|x-x'|\approx \diam(K_{\omega_*})=r_{\omega_*} \approx 3^{-(1+k_*c)}$. Since $g$ sends $x,x'$ into a specific cell in $\{F_{\j}:\j\in J_{1+(k_*-1)c}\}$, we have 
    \[
        |g(x)-g(x')| \leq \max\{\diam(F_{\j}): \j\in J_{1+(k_*-1)c}\} \approx 3^{-1-(k_*-1)c} \approx |x-x'|.
    \]
    So $g$ is Lipschitz.
\end{example}

\bigskip
\noindent{\bf Acknowledgements.}
The research of the authors is partially supported by NSFC grants 12371089, 12501113, 12531004, and the Fundamental Research Funds for the Central Universities of China grant 2024FZZX02-01-01.

\small
\bibliographystyle{amsplain}

\end{document}